\documentclass[a4paper,reqno,twoside,12pt]{amsart}
\usepackage{amssymb,amsmath}
\usepackage[dvips]{color}
\usepackage[dvips,final]{graphics}
\usepackage{hyperref}
\usepackage{amscd}
\usepackage[latin1]{inputenc}

\usepackage{a4wide}
\theoremstyle{plain}
\newtheorem{theorem}{Theorem}[section]      
\newtheorem{proposition}[theorem]{Proposition}    
\newtheorem{lemma}[theorem]{Lemma}            

\theoremstyle{definition}               
\newtheorem{definition}[theorem]{Definition}  
\theoremstyle{remark}                   
\newtheorem{remark}[theorem]{Remark}

\newcommand{\re}{\mathbb{R}}

\numberwithin{equation}{section}
\begin{document}

\title[Energy estimates for the half-Laplacian]
{Energy estimates and 1-D symmetry for nonlinear
equations involving the half-Laplacian}

\thanks{Both authors were supported by grants
MTM2008-06349-C03-01 (Spain) and 2009SGR-345 (Catalunya). The second author was partially supported by University of Bologna (Italy), funds for selected research topics.}
\author[Xavier Cabr\'e]{Xavier Cabr\'e}
\address{ICREA and
Universitat Polit\`ecnica de Catalunya, Diagonal 647 \\
Departament de Mate-\-m\`atica Aplicada 1 \\
08028 Barcelona (Spain) } \email{xavier.cabre@upc.edu}
\author[Eleonora Cinti]{Eleonora Cinti}
\address{
Universit\`a di Bologna\\
Dipartimento di Matematica\\
Piazza di Porta San Donato 5\\
40126 Bologna (Italy) and Universitat Polit\`ecnica de Catalunya, Diagonal 647 \\
Departament de Matem\`atica Aplicada 1\\
08028 Barcelona (Spain) } \email{eleonora.cinti@upc.edu}


\begin{abstract}
We establish sharp energy estimates for some solutions, such as
global minimizers, monotone solutions and saddle-shaped solutions,
of the fractional nonlinear equation $(-\Delta)^{1/2} u=f(u)$ in $\re^n$. Our energy 
estimates hold for every nonlinearity $f$ and are sharp since they are
optimal for one-dimensional solutions, that is, for solutions
depending only on one Euclidian variable.

As a consequence, in dimension $n=3$, we deduce the
one-dimensional symmetry of every global minimizer and of every
monotone solution. This result is the analog of a conjecture of De
Giorgi on one-dimensional symmetry for the classical equation
$-\Delta u=f(u)$ in $\re^n$.
\end{abstract}

\maketitle

\section{Introduction and results}
In this paper we establish sharp energy estimates for some solutions of
the fractional nonlinear equation
\begin{equation} \label{eq1}
 (-\Delta)^{1/2} u=f(u)\quad \mbox{in}\;\re^n,\end{equation}
where $f:\re\rightarrow\re$ is a $C^{1,\beta}$ function with
$0<\beta<1$. 
When $f$ is a balanced bistable nonlinearity, for instance when
$f(u)=u-u^3$, we call
equation (\ref{eq1}) of Allen-Cahn type by the analogy with the
corresponding equation involving the Laplacian instead of the
half-Laplacian,
\begin{equation}\label{eq-degiorgi}
-\Delta u=u-u^3\quad \mbox{in}\;\re^n.\end{equation}

In 1978 De Giorgi conjectured that the level sets of every
bounded solution of
(\ref{eq-degiorgi}) which is monotone in one direction, must be hyperplanes, at least if $n\leq 8$.
That is, such solutions depend only on one Euclidian variable. The
conjecture has been proven to be true in dimension $n=2$ by
Ghoussoub and Gui \cite{GG} and in dimension $n=3$ by Ambrosio and
the first author \cite{AC}. For $4\leq n\leq 8$, if $\partial_{x_n}u>0$,
and assuming the additional condition
$$\lim_{x_n \rightarrow \pm \infty}u(x',x_n)=\pm 1\quad \mbox{for all}\;x'\in \re^{n-1},$$ it has been
established by Savin \cite{S}. Recently a counterexample to the
conjecture for $n\geq 9$ has been announced by del Pino, Kowalczyk
and Wei \cite{dPKW}.

In this paper (see Theorem \ref{degiorgi} below), we establish the
one-dimensional symmetry of bounded monotone solutions of (\ref{eq1}) in
dimension $n=3$, that is, the analog of the conjecture of De Giorgi
for the half-Laplacian in dimension 3. We recall that
one-dimensional (or 1-D) symmetry for bounded stable solutions of
(\ref{eq1}) in dimension $n=2$ has been proven by the first author and
Sol\`a-Morales \cite{C-SM}. The same result in dimension $n=2$ for
the other fractional powers of the Laplacian, i.e., for the
equation
$$(-\Delta)^s u=f(u)\quad \mbox{in}\;\re^2,\quad \mbox{with}\;0<s<1,$$ has been established by the first author and Sire \cite{C-Si1, C-Si2} and by Sire and Valdinoci
\cite{SV}.

A crucial ingredient in our proof of 1-D symmetry in $\re^3$ is a sharp
energy estimate for global minimizers and for monotone solutions, that we state in Theorems
\ref{energy-est} and \ref{energy-dim3} below. It is interesting to note that our method to prove the energy estimate
also applies to the case of saddle-shaped solutions in $\re^{2m}.$ These solutions are not
global minimizers in general (this is indeed the case in dimensions $2m\leq 6$ by a result of the second author \cite{Cinti}), but they are
minimizers under perturbations
vanishing on a suitable subset of $\re^{2m}$. We treat these solutions and their corresponding energy estimate at the end of this introduction.

To study the nonlocal problem (\ref{eq1}) we realize it as
a local problem in $\re^{n+1}_+$ with a nonlinear Neumann
condition on $\partial \re_{+}^{n+1}=\re^n$. More precisely, if $u=u(x)$ is a function defined on
$\re^n$, we consider its harmonic extension $v=v(x,\lambda)$
in $\re^{n+1}_+=\re^n\times(0,+\infty)$. It is well known (see \cite{C-SM,CS}) that $u$ is a solution of
(\ref{eq1}) if and only if $v$ satisfies
\begin{equation}\label{eq2}
\begin{cases}
\Delta v=0& \text{in}\; \re_{+}^{n+1},\\
- \partial_{ \lambda}v=f(v)& \text{on}\; \re^{n}=\partial
\re_{+}^{n+1}.
\end{cases}
\end{equation}

Problem (\ref{eq2}) allows to introduce the notions of \textit{energy}
 and \textit{global minimality} for a solution $u$ of problem (\ref{eq1}).
Consider the cylinder
$$C_R=B_R \times (0,R)\subset \re^{n+1}_+,$$ where $B_R$ is the ball of radius $R$ centered at 0 
in $\re^n$.
We consider the energy functional
\begin{equation}\label{energia}
{\mathcal E}_{C_R}(v)=\int_{C_R}\frac{1}{2}|\nabla v|^2dx d\lambda
+ \int_{B_R}G(v)dx, \end{equation}
whose Euler-Lagrange equation is problem \eqref{eq2}. The potential $G$, defined up to an additive constant, is given by
$$G(v)=\int_v^1 f(t)dt.$$

Using the energy functional \eqref{energia}, we introduce the notions of \textit{global minimizer} and of \textit{layer solution} of \eqref{eq1}. We call layer solutions of \eqref{eq1} those bounded solutions that are monotone increasing, say from $-1$ to $1$, in one of the $x$-variables. After rotation, we can suppose that the direction of monotonicity is the $x_n$-direction, as in point c) of the following definition.
\begin{definition}\label{layer}
a) We say that a bounded $C^1(\overline{\re_+^{n+1}})$
function $v$ is a {\it global minimizer} of $(\ref{eq2})$ if, for
all $R>0$,
$${\mathcal E}_{C_R}(v)\leq {\mathcal E}_{C_R}(w)$$ for
every $C^{1}(\overline{\re_+^{n+1}})$ function $w$ such that
$v\equiv w$ in $\re_+^{n+1}\setminus \overline{C_R}$. 

b) We say that a bounded $C^1$ function $u$ in $\re^{n}$ is a {\it
global minimizer} of $(\ref{eq1})$ if its harmonic extension $v$
is a global minimizer of (\ref{eq2}). 

c) We say that a
bounded function $u$ is a \emph{layer solution} of
(\ref{eq1}) if $\partial_{x_n}u>0$ in $\re^n$ and
\begin{equation}\label{hp-limit}\lim_{x_n \rightarrow \pm
\infty}u(x',x_n)=\pm 1\quad \mbox{for every}\:x'\in
\re^{n-1}.\end{equation}
\end{definition}

Note that the functions $w$ in point a) of Definition \ref{layer} need to agree with the solution $v$ on the lateral boundary and on the top of the cylinder $C_R$, but not on 
its bottom. Since it will be useful in the sequel, we denote the lateral and top parts
of the boundary of $C_R$ by
$$\partial^+C_R=\partial C_R \cap \{\lambda>0\}.$$ 

In some references, global minimizers are called ``local
minimizers'', where local stands for the fact that the energy is
computed in bounded domains.

We recall that every layer solution is a global minimizer
(see Theorem 1.4 in \cite{C-SM}). 

Our main result is the following energy estimate for global
minimizers of problem (\ref{eq1}). Given a bounded function $u$
defined on $\re^n$, set 
\begin{equation}\label{c_u}
c_u=\min\{G(s):\inf_{\re^n} u\leq s\leq \sup_{\re^n}
u\}.\end{equation}

\begin{theorem}\label{energy-est}
Let $f$ be any $C^{1,\beta}$ nonlinearity, with $\beta\in
(0,1)$, and $u\in L^{\infty}(\re^{n})$ be a global minimizer of
(\ref{eq1}). Let $v$ be the harmonic extension of $u$ in
$\re^{n+1}_+$.

Then, for all $R>2$,
\begin{equation}\label{energy}
\int_{C_R}\frac{1}{2}|\nabla v|^2dx
d\lambda+\int_{B_R}\{G(u)-c_u\}dx \leq CR^{n-1}\log
R,\end{equation} where $c_u$ is defined by (\ref{c_u}) and $C$ is
a constant depending only on $n$, $||f||_{C^1([\inf u,\sup u])}$, and $||u||_{L^\infty(\re^n)}$.
In particular, we have that
\begin{equation}\label{kinetic}
\int_{C_R}\frac{1}{2}|\nabla v|^2dx d\lambda \leq CR^{n-1}\log
R.\end{equation}
\end{theorem}

As a consequence, (\ref{energy}) and (\ref{kinetic}) also hold for
layer solutions. We stress that this
energy estimate is sharp because it is optimal for 1-D solutions, in the sense that for some explicit 1-D solutions the energy is also bounded below by  $cR^{n-1}\log R$, for some constant $c>0$, when they are seen as solutions in $\re^n$ (see Remark~\ref{Rk2.2} below and
section 2.1 of \cite{C-SM}).

In dimensions $n=1$ and $n=2$ estimate \eqref{energy} was established by
the first author and Sol\`a-Morales in \cite{C-SM}. 

In dimension $n=3$, the energy estimate \eqref{energy} holds also for
monotone solutions which do not satisfy the limit assumption \eqref{hp-limit}.
These solutions are minimizers in some sense to be explained
later, but, in case that they exist, they are not known to be global minimizers as defined
before.

\begin{theorem}\label{energy-dim3}
Let $n=3$, $f$ be any $C^{1,\beta}$ nonlinearity with $\beta
\in (0,1)$, and $u$ be a bounded solution of (\ref{eq1}) such that
$\partial_{e}u>0$ in $\re^3$ for some direction $e\in \re^3$,
$|e|=1$. Let $v$ be its harmonic extension in $\re^{4}_+$.

Then, for all $R>2$,
\begin{equation}\label{en3}
\int_{C_R}\frac{1}{2}|\nabla v|^2dx
d\lambda+\int_{B_R}\{G(u)-c_u\}dx \leq CR^{2}\log R,\end{equation}
where $c_u$ is defined by (\ref{c_u}) and $C$ is a constant
depending only on $||f||_{C^1([\inf u,\sup u])}$ and $||u||_{L^\infty(\re^3)}$.
\end{theorem}

In dimension $n=3$, Theorems \ref{energy-est} and
\ref{energy-dim3} lead to the 1-D symmetry of global minimizers and of
monotone solutions to problem (\ref{eq1}).

\begin{theorem}\label{degiorgi}
Let $n=3$ and $f$ be any $C^{1,\beta}$ nonlinearity with
$\beta \in (0,1)$. Let $u$ be either a bounded global minimizer of
\eqref{eq1}, or a bounded solution of \eqref{eq1} monotone in some direction
$e\in \re^3$, $|e|=1$.

Then,
$u$ depends only on
one variable, i.e., there exists $a \in \re^3$ and $g: \re
\rightarrow \re $, such that $u(x)=g(a\cdot x)$ for all $x \in
\re^3$. Equivalently, the level sets of $u$ are planes.\end{theorem}

To prove 1-D symmetry, we use a standard Liouville type argument which requires an appropriate estimate
for the kinetic energy. By a result of Moschini \cite{mosch} (see Proposition \ref{moschini} in section 6 below),
our energy estimate in $\re^3$,
$$\int_{C_R}|\nabla v|^2dx d\lambda\leq CR^2 \log R,$$ allows to use such Liouville type result and deduce 1-D symmetry in $\re^3$ for global minimizers and for solutions
monotone in one direction.

\begin{remark}\label{sharp} As a consequence of Theorem \ref{degiorgi}, we obtain that for all $R>2$,
\begin{equation}\label{poten}
\int_{B_R}G(v(x,0))dx\leq CR^{n-1}\quad \mbox{if}\:\:1\leq n\leq 3,
\end{equation} if $v$ is a bounded global minimizer or a bounded monotone solution of \eqref{eq2}. This was proven in \cite{C-SM} for $n=1$ and $n=2$. For $n=3$, \eqref{poten} follows from the $n=1$ case after using Theorem \ref{degiorgi}. In dimension $n\geq 4$ we do not know if the potential energy can be bounded by $CR^{n-1}$ (instead of $CR^{n-1}\log R$) as in \eqref{poten}.
\end{remark}
In our next paper, using similar techniques, we establish sharp energy
estimates for the other fractional powers of the Laplacian. More
precisely, we prove that if $u$ is a bounded global minimizer
of 
\begin{equation}\label{eqs}(-\Delta)^su=f(u)\;\;\mbox{in}\;\re^n,\;\;\mbox{with}\;0<s<1,
\end{equation} then
the following energy estimate holds:
$$\mathcal{E}_{s,C_R}(u)\leq CR^{n-2s}\quad \mbox{for}\;0<s<\frac{1}{2},$$
$$\mathcal{E}_{s,C_R}(u)\leq CR^{n-1}\quad \mbox{for}\;\frac{1}{2}<s<1.$$
 Here the energy functional
is defined using a local formulation in $\re^{n+1}_+$ of problem
(\ref{eqs}), found by Caffarelli and Silvestre in \cite{CS}. If
$1/2<s<1$ then $\mathcal{E}_{s,C_R}(u)\leq CR^{n-1}$; in this case we
can deduce 1-D symmetry for global minimizers and monotone
solutions in dimension $n=3$.

Back to the case $s=1/2$, we have two different proofs of the energy estimate $CR^{n-1}\log R$. 

The first one is very simple but applies only to Allen-Cahn type nonlinearities (such as $f(u)=u-u^3$) and to monotone solutions satisfying the limit assumption \eqref{hp-limit} or the more general \eqref{limit} below. We present this very simple proof in section~2. It was found by Ambrosio and the first author \cite{AC} to prove the optimal energy estimate for $-\Delta u=u-u^3$ in $\re^n$.

Our second proof applies in more general situations and will lead to Theorems \ref{energy-est} and \ref{energy-dim3}. It is based on controlling the $H^1(\Omega)$-norm of a function by its fractional Sobolev norm $H^{1/2}(\partial \Omega)$ on the boundary.

Let us recall the definition of the $H^{1/2}(A)$ norm,
 where $A$ is either a Lipschitz open set of $\re^n$,
or $A=\partial \Omega$ and $\Omega$ is a Lipschitz open set of
$\re^{n+1}$. It is given by
$$||w||^2_{H^{1/2}{(A)}}=||w||^2_{L^2(A)}+\int_{A}\int_{A}\frac{|w(z)-w(\overline{z})|^2}
{|z-\overline{z}|^{n+1}}d\sigma_z
d\sigma_{\overline{z}}.$$ In our proof we will have
$A=\partial C_R\subset \re^{n+1}$, the boundary of the cylinder $\Omega=C_R$.

In the proof of Theorem \ref{energy-est} a crucial point will be
the following well known result. If $w$ is a function in
$H^{1/2}(\partial \Omega)$, where $\Omega$ is a bounded subset of
$\re^{n+1}$ with Lipschitz boundary, then the harmonic extension
$\overline{w}$ of $w$ in $\Omega$ satisfies:
\begin{equation}\label{H-est}
\int_{\Omega}|\nabla \overline{w}|^2 \leq C(\Omega)
||w||^2_{H^{1/2}(\partial \Omega)}.\end{equation} For the sake of
completeness (and since the proof will be important in our next paper \cite{C-Cinti2}), we will recall a proof of this result in
section 3 (see Proposition \ref{extension}).

To prove the sharp energy estimate for a global minimizer $v$ in $\re^{n+1}_+$, we will bound its energy in the cylinder $C_R=B_R\times (0,R)$, using \eqref{H-est}, by that of the harmonic extension $\overline w$ in $C_R$ of a well chosen function $w$ defined on $\partial C_R$. This function $w$ must agree with $v$ on $\partial^+C_R$ (the lateral and top boundaries of $C_R$), while it will be identically $1$ on the portion $B_{R-1}\times \{0\}$ of the bottom boundary. In this way, it will not pay potential energy in this portion of the bottom boundary.

By \eqref{H-est}, we will need to control $||w||_{H^{1/2}(\partial C_R)}$. After rescaling $\partial C_R$ to $\partial C_1$, we will control the $H^{1/2}$-norm of $w$ using the following key result. It gives a bound on the $H^{1/2}$-norm of functions on $A$ which satisfy a certain
gradient pointwise bound related with the distance to a Lipschitz subset $\Gamma$
of $A$. We will apply it in the sets
$$A=\partial C_1\quad \mbox{and}\quad \Gamma=\partial B_1\times \{\lambda=0\},$$
with a small parameter $\varepsilon=1/R$. Examples in which the following theorem applies are, among many others, $A=B_1\subset \re^n$ the unit ball and $\Gamma=B_1\cap \{x_n=0\}$, and 
also $A=B_1\subset \re^n$ and $\Gamma=\partial B_r$ for some $r\in (0,1)$.
\begin{theorem}\label{key}
Let $A$ be either a bounded Lipschitz domain in $\re^n$ or $A=\partial
\Omega$, where $\Omega$ is a bounded open set of $\re^{n+1}$ with
Lipschitz boundary. Let $M \subset A$ be an open set (relative to $A$) with
Lipschitz boundary (relative to $A$) $\Gamma \subset A$. Let $\varepsilon\in (0,1/2)$. 

Let $w:A \rightarrow \re$ be a Lipschitz function such that, for almost
every $x\in A$, 
 \begin{equation}\label{bound_w}|w(x)|\leq c_0\end{equation} and 
 \begin{equation}\label{bound_grad}
\displaystyle |D w(x)|\leq c_0\min\left\{\frac{1}{\varepsilon},\frac{1}{\mbox{\rm{dist}}(x,\Gamma)}\right\},\end{equation}
where $D$ are all tangential derivatives to $A$, $\mbox{\rm{dist}}(x,\Gamma)$ is the distance 
from the point
$x$ to the
set $\Gamma$ (either in $\re^n$ or in $\re^{n+1}$), and $c_0$ is a positive constant.

Then,
\begin{equation}\label{logB}
||w||^2_{H^{1/2}{(A)}}=||w||^2_{L^2(A)}+\int_{A}\int_{A}\frac{|w(z)-w(\overline
z)|^2}{|z-\overline z|^{n+1}}d\sigma_z d\sigma_{\overline z}\leq c_0^2C
|\log \varepsilon|,
\end{equation} where $C$ is a positive constant depending only on
$A$ and $\Gamma$.
\end{theorem}

As we said, we will use this result with $A=\partial C_1$ and
$\Gamma=\partial B_1 \times \{\lambda=0\}$. Thus, in this case the
constant $C$ in \eqref{logB} only depends on the
dimension $n$.
The gradient estimate \eqref{bound_grad}, after rescaling $\partial C_R$ to $\partial C_1$ and taking $\varepsilon=1/R$, will follow from the bound
\begin{equation}\label{grad_CSM}|\nabla v(x,\lambda)|\leq \frac{C}{1+\lambda}\quad \mbox{for
all}\;x\in \re^n\;\mbox{and}\;\lambda\geq 0,
\end{equation} 
satisfied by every bounded solution $v$ of \eqref{eq2}. Here the constant $C$ depends only on $n$, $||f||_{C^1}$, and $||v||_{L^\infty(\re^{n+1}_+)}$.
For $\lambda \geq 1$, \eqref{grad_CSM} follows immediately from the fact that $v$ is bounded and harmonic in $B_{\lambda}(x,\lambda)\subset \re^{n+1}_+$. For $\lambda \leq 1$, estimate \eqref{grad_CSM}  for bounded solutions of 
the nonlinear Neumann problem \eqref{eq2} was proven in Lemma 2.3 of \cite{C-SM}.

Our method to prove sharp energy estimates also applies to solutions which are 
minimizers under perturbations
vanishing on a suitable subset of $\re^n$, even if they are not in general global minimizers as defined before. An important example of this are some \emph{saddle-shaped solutions} (or
\emph{saddle solutions} for short) of
$$(-\Delta)^{1/2}u=f(u)\quad \mbox{in}\;\;\re^{2m}.$$ The existence and qualitative properties of
these solutions have been studied by the second author in \cite{Cinti}.
For equations of Allen-Cahn type involving the Laplacian,
$-\Delta u=f(u)$, saddle solutions have been studied in \cite{CT,CT2}.

Saddle solutions are even with respect to the coordinate axes and
odd with respect to the Simons cone, which is defined as follows.
For $n=2m$ the Simons cone ${\mathcal C}$ is given by
$$\mathcal{C}=\{x\in
\re^{2m}:x_1^2+...+x_m^2=x_{m+1}^2+...+x_{2m}^2\}.$$
We define two new variables $$s=\sqrt{x_1^2+\dots + x_m^2} \quad
\mbox{ and }\quad t=\sqrt{x_{m+1}^2+\dots + x_{2m}^2},$$ for which
the Simons cone becomes ${\mathcal C}=\{s=t\}$.

The existence of saddle solutions of \eqref{eq1} has been proven in \cite{Cinti} under the following hypotheses on $f$:
\begin{equation}\label{h11}
f \; \text{ is odd};\end{equation}
\begin{equation}\label{h22}
 G\geq 0=G(\pm 1)\; {\rm in\, } \re, \,{\rm and }\, G>0\; {\rm in }\, (-1,1);\end{equation}
\begin{equation}\label{h33}  f' \; \text{ is decreasing in } (0,1).\end{equation}

Note that, if \eqref{h11} and \eqref{h22} hold, then $f(0)=f(\pm 1)=0.$
Conversely, if $f$ is odd in $\re$, positive with $f'$ decreasing
in $(0,1)$ and negative in $(1,\infty)$ then $f$ satisfies \eqref{h11},
\eqref{h22} and \eqref{h33}. Hence, the nonlinearities $f$ that we consider are
of ``balanced bistable type", while the potentials $G$ are of
``double well type". Our three assumptions \eqref{h11}, \eqref{h22}, \eqref{h33} are
satisfied by the scalar Allen-Cahn type equation
\begin{equation*}
(-\Delta)^{1/2} u= u-u^3.
\end{equation*}
In this case we have that $G(u)=(1/4)(1-u^2)^2$. The three hypotheses also hold for the equation
$(-\Delta)^{1/2} u=\sin (\pi u)$, for which $G(u)=(1/\pi)(1+\cos
(\pi u))$.

The following result states the existence of at least one saddle solution for which our 
sharp energy estimate holds.

\begin{theorem}\label{saddle} Let $f$ be a $C^{1,\beta}$ function for some $0<\beta<1$, satisfying \eqref{h11}, \eqref{h22}, and \eqref{h33}.
Then, there exists a {\it saddle solution} $u$ of $(-\Delta)^{1/2}
u=f(u)$ in $\re^{2m}$, i.e., a bounded solution $u$ such that

\noindent
(a) $u$ depends only on the variables $s$ and $t$. We write
$u=u(s,t)$; 

\noindent (b) $u>0$ for $s>t$; 

\noindent (c) $u(s,t)=-u(t,s)$.

 Moreover, $|u|< 1$ in $\re^{2m}$ and for every $R>2$,
$$\mathcal E_{C_R}(v)\leq CR^{2m-1}\log R,$$
where $v$ is the harmonic extension of $u$ in
$\re^{2m+1}_+$ and $C$ is a constant depending only on $m$ and $||f||_{C^1([-1,1])}$.
\end{theorem}

Observe that the
saddle solution of the theorem satisfies the same optimal energy estimate as global minimizers do, that is, $CR^{n-1}\log R=CR^{2m-1}\log R$, even that in low dimensions it is known \cite{Cinti} that saddle solutions are not global minimizers. Indeed saddle solutions are not stable in dimension 2 (by a result of the first author and Sol\`a-Morales \cite{C-SM}) and in dimensions 4 and 6 (by a result of the second author \cite{Cinti}). As we will explain in the last section, some saddle solutions are minimizers under perturbations vanishing on the Simons cone,  
and this will be enough to prove that they satisfy the sharp energy estimate.

The paper is organized as follows:
\begin{itemize}
\item In section 2 we prove the energy estimate for layer
solutions of Allen-Cahn type equations, using a simple argument
found by Ambrosio and the first author~\cite{AC}.

\item In section 3 we give the proof of the extension theorem
 and of the key Theorem \ref{key}.

\item In section 4 we prove energy estimate \eqref{energy} for
global minimizers and for every nonlinearity $f$, that is, Theorem \ref{energy-est}.

\item In section 5 we establish energy estimates for monotone solutions
in $\re^3$, Theorem \ref{energy-dim3}.

\item In section 6 we prove the 1-D symmetry result, that is,
Theorem \ref{degiorgi}.

\item In section 7 we prove the energy estimate for saddle solutions, Theorem \ref{saddle}.
\end{itemize}

\section{Energy estimate for monotone solutions of Allen-Cahn type equations}
In this section we consider potentials $G$ which satisfy
hypothesis \eqref{h22}, that is, $G\geq 0=G(\pm1)$ in $\re$ and $G>0$ in
$(-1,1)$. In the sequel we consider the energy $\mathcal{E}_{C_R}$
defined by
$$\mathcal{E}_{C_R}(v)=\int_{C_R}\frac{1}{2}|\nabla v|^2dx d\lambda
+ \int_{B_R}G(v)dx.$$ In general, it can be defined up to an
additive constant $c$ in the potential $G(v)-c$, but in this case, by the assumption \eqref{h22} on
$G$, we take $c=0$.
\begin{theorem}\label{energy-layer}
Let $f$ be a $C^{1,\beta}$ function, for some $0<\beta<1$, 
satisfying \eqref{h22}, where $G'=-f$. Let $u$ be a bounded solution of problem
\eqref{eq1} in $\re^n$, with $|u|<1$ in $\re^n$, and let $v$ be the harmonic
extension of $u$ in $\re^{n+1}_+$. Assume that
\begin{equation}\label{hp-monot}u_{x_n}>0 \;\;\mbox{in}\;\; \re^n
\end{equation} and
\begin{equation}\label{limit}
 \lim_{x_n\rightarrow
+\infty}u(x',x_n)=1\;\;\mbox{for all}\:\;x'\in
\re^{n-1}.\end{equation} 

Then, for every $R>2$,
$$\int_{C_R}\frac{1}{2}|\nabla v|^2dx d\lambda \leq {\mathcal E}_{C_R}(v)\leq CR^{n-1}\log R,$$
for some constant $C$ depending only on $n$ and $||f||_{C^1([-1,1])}$.
\end{theorem}
\begin{remark}\label{Rk2.2}
This energy estimate in dimension $n=1$ has been proven by the first author and
Sol\`a-Morales \cite{C-SM}, using the gradient bound
\begin{equation}\label{grad-layer}
|\nabla v(x,\lambda)|\leq \frac{C}{1+|(x,\lambda)|}\;\;\mbox{for
all}\;\;x\in \re\;\mbox{and}\;\lambda\geq 0,\end{equation} (see estimate (1.14) of \cite{C-SM}). 
Indeed, we next see that
(\ref{grad-layer}) leads to
$$\int_{C_R}|\nabla v|^2 dx d\lambda\leq C \log R$$
and also
\begin{equation}\label{infinit-cyl}\int_0^{+\infty}d\lambda\int_{B_R}dx |\nabla v|^2  \leq C
\log R.\end{equation} That is, for  $n=1$, the energy
estimate holds not only in the cylinder $C_R$, but also in the infinite
cylinder $B_R\times (0,+\infty)$.
Let us mention that for the explicit layer solutions in section 2.1 of \cite{C-SM}, the upper bound $C(1+|(x,\lambda)|)^{-1}$ for $|\nabla v|$ is also a lower bound for $|\nabla v|$, modulo a smaller multiplicative constant. As a consequence, the following computation shows that
the two previous upper bounds $\log R$ are also lower bounds for the Dirichlet energy 
after multiplying $\log R$ by a smaller constant.

Estimate \eqref{infinit-cyl}
holds, indeed:
\begin{eqnarray*}
\int_0^{+\infty}d\lambda \int_{-R}^R dx |\nabla v|^2 &\leq& C
\int_0^{+\infty} d\lambda \int_{-R}^R dx\frac{1}{1+x^2+\lambda^2}
\\
&\leq & C \int_{-R}^R dx \int_0^{+\infty}d\lambda
\frac{1}{(1+x)^2}\cdot\frac{1}{1+\left(\frac{\lambda}{1+x}\right)^2}\\
&\leq&
C\int_{-R}^R\left[\frac{1}{1+x}\arctan{\frac{\lambda}{1+x}}\right]_{\lambda=0}^{\lambda=+\infty}dx
\\
&\leq& C\int_{-R}^R \frac{\pi}{2} \frac{1}{1+x}dx\leq C \log
R.\end{eqnarray*}

In higher dimensions, an analog of
(\ref{grad-layer}) is not available and therefore we need another
method to prove Theorem \ref{energy-layer}.
\end{remark}

\begin{proof}[Proof of Theorem \ref{energy-layer}]
We follow an argument found by
Ambrosio and the first author \cite{AC} to prove the energy estimate for layer
solutions of the analog problem $-\Delta u=f(u)$ in $\re^n$. It is based on
sliding the function $v$, which is the harmonic extension of the
solution $u$, in the direction $x_n$.

Consider the function
$$v^t(x,\lambda):=v(x',x_n+t,\lambda)$$ defined for $(x,\lambda)=(x',x_n,\lambda)\in
\re^{n}\times [0,+\infty)$, where $t\in \re$. For each $t$ we have
\begin{equation}
\begin{cases}
\Delta v^t=0& \text{in}\; \re_{+}^{n+1},\\
-\partial_{\lambda}v^t=f(v^t)& \text{on}\; \re^{n}=\partial
\re_{+}^{n+1}.
\end{cases}
\end{equation}
Moreover, as stated in \eqref{grad_CSM}, the following bounds hold:
\begin{equation}\label{bound-t}|v^t|\leq 1\quad \mbox{and}\quad |\nabla v^t|\leq
\frac{C}{1+\lambda}.\end{equation} Throughout the proof, $C$ will denote
different positive constants which depending only on $n$ and $||f||_{C^1([-1,1])}$.

A simple compactness argument implies that
\begin{equation}\label{lim}\lim_{t\rightarrow +\infty}\{ |v^t-1|+|\nabla v^t|\}=0\end{equation} uniformly in compact sets of $\overline{\re_+^{n+1}}$.
Indeed, arguing by contradiction, assume that there exist $R>0$, $\varepsilon>0$, and a sequence $t_m\rightarrow \infty$ such that
\begin{equation}\label{ass}||v^{t_m}-1||_{L^{\infty}(C_R)}+||\nabla v^{t_m}||_{L^{\infty}(C_R)}\geq \varepsilon\end{equation}
for every $m$, where $C_R=B_R\times (0,R)$. Since $v^{t_m}$ are all solutions of \eqref{eq2} in all the halfspace, the regularity results in \cite{C-SM} give $C^2_{{\rm loc}}(\overline{\re_+^{n+1})}$ estimates for $v^{t_m}$ uniform in $m$. Thus, there exists a subsequence that converges in $C^2_{\text{{\rm loc}}}(\overline {\re^{n+1}_+})$ to a bounded harmonic function $v^{\infty}$. By hypothesis \eqref{limit}, $v^{\infty}\equiv 1$ on $\partial \re^{n+1}_+$, and thus by the maximum principle, $v^{\infty}\equiv 1$ in all of $\re^{n+1}_+$. This contradicts \eqref{ass}, by $C^1$ convergence in compact sets of $v^{t_m}$ towards $v^{\infty} \equiv 1$.

Denoting the derivative of $v^t(x,\lambda)$ with respect to $t$ by
$\partial_t v^t(x,\lambda)$, we have
$$\partial_t v^t(x,\lambda)=v_{x_n}(x',x_n+t,\lambda)>0\quad
\mbox{for all}\;\;x \in \re^n,\;\lambda\geq 0.$$ 
Note that $v_{x_n}>0$, since $v_{x_n}$ is the harmonic extension of the bounded function $u_{x_n}>0$.
We consider the
energy of $v^t$ in the cylinder $C_R=B_R\times (0,R)$,
$${\mathcal E}_{C_R}(v^t)=\int_{C_R}\frac{1}{2}|\nabla v^t|^2dx
d\lambda + \int_{B_R}G(v^t)dx.$$ Note that, by 
(\ref{lim}), we have
\begin{equation}\label{lim-en-t}\lim_{t\rightarrow
+\infty}\mathcal{E}_{C_R}(v^t)=0.
\end{equation} 

Next, we bound the derivative
of $\mathcal{E}_{C_R}(v^t)$ with respect to $t$. We use that $v^t$
is a solution of problem (\ref{eq2}), the bound \eqref{bound-t} for $|v^t|$ and
$|\nabla v^t|$, and the crucial fact that $\partial_t v^t>0$. Let
$\nu$ denote the exterior normal to the lateral boundary $\partial
B_R\times (0,R)$ of the cylinder $C_R$. We have
\begin{eqnarray*}
\partial_t\mathcal{E}_{C_R}(v^t)&=&\int_0^R d\lambda
\int_{B_R}dx\ \nabla v^t\cdot \nabla (\partial_t v^t)
+\int_{B_R}dx\ G'(v^t)\partial_t v^t\\
&=& \int_0^Rd\lambda \int_{\partial B_R}d\sigma\ \frac{\partial
v^t}{\partial \nu}
\partial_t v^t+\int_{B_R\times \{\lambda=R\}}dx\ \frac{\partial
v^t}{\partial\lambda}
\partial_t v^t\\
&\geq& -C\int_0^R \frac{d\lambda}{1+\lambda}\int_{\partial
B_R}d\sigma\ \partial_t v^t -\frac{C}{R}\int_{B_R\times
\{\lambda=R\}}dx\ \partial_t v^t.\end{eqnarray*} Hence, for
every $T>0$, we have
\begin{eqnarray*}
\mathcal{E}_{C_R}(v)&=&\mathcal{E}_{C_R}(v^T)-\int_0^T dt\ \partial_t
\mathcal{E}_{C_R}(v^t) \\
&\leq& \mathcal{E}_{C_R}(v^T)+ C\int_0^Tdt \int_0^R
\frac{d\lambda}{1+\lambda}\int_{\partial B_R}d\sigma\ \partial_t v^t
+\frac{C}{R}\int_0^T dt\int_{B_R\times
\{\lambda=R\}}dx\ \partial_t v^t  \\
&=& \mathcal{E}_{C_R}(v^T)+ C\int_0^R \frac{d\lambda}{1+\lambda}
\int_{\partial B_R}d\sigma\int_0^T dt\ \partial_t v^t 
+\frac{C}{R}\int_{B_R\times
\{\lambda=R\}}dx \int_0^T dt\ \partial_t v^t \\
&=& \mathcal{E}_{C_R}(v^T)+ C\int_0^R \frac{d\lambda}{1+\lambda}
\int_{\partial B_R}d\sigma\ (v^T-v^0)
+\frac{C}{R}\int_{B_R\times
\{\lambda=R\}}dx\ (v^T-v^0)\\
 &\leq&\mathcal{E}_{C_R}(v^T) + CR^{n-1}\log R +
 CR^{n-1}.\end{eqnarray*}
 Letting $T \rightarrow + \infty$ and using \eqref{lim-en-t}, we obtain the desired estimate.
\end{proof}

\section{$H^{1/2}$ estimate}

In this section we recall some definitions and properties  about
the spaces $H^{1/2}(\re^n)$ and $H^{1/2}(\partial \Omega)$, where
$\Omega$ is a bounded subset of $\re^{n+1}$ with Lipschitz
boundary $\partial \Omega$ (see
\cite{LM}).

$H^{1/2}(\re^n)$ is the space of functions $u\in L^2(\re^n)$ such
that
$$\int_{\re^n}\int_{\re^n}\frac{|u(x)-u(\overline{x})|^2}{|x-\overline{x}|^{n+1}}dx
d\overline{x}< +\infty,$$ equipped with the norm

$$||u||_{H^{1/2}(\re^n)}=\left(||u||^2_{L^2(\re^n)}+ \int_{\re^n}\int_{\re^n}\frac{|u(x)-u(\overline{x})|^2}{|x-\overline{x}|^{n+1}}dx
d\overline{x}\right)^\frac{1}{2}.$$

Let now $\Omega$ be a bounded subset of $\re^{n+1}$ with
Lipschitz boundary $\partial \Omega$. To define $H^{1/2}(\partial
\Omega)$, consider an atlas $\{(O_j,\varphi_j);j=1,...,m\}$ where
$\{O_j\}$ is a family of open bounded sets in $\re^{n+1}$ such
that $\{O_j\cap\partial \Omega; j=1,...,m\}$ cover $\partial \Omega$. The functions
$\varphi_j$ are bilipschitz diffeomorphisms such that
\begin{itemize}
\item $\varphi_j:O_j\rightarrow U:=\{(y,\mu)\in
\re^{n+1}:|y|<1,\:-1<\mu<1\}$, \item

$\varphi_j:O_j\cap \Omega \rightarrow U^+:=\{(y,\mu)\in
\re^{n+1}:|y|<1,\:0<\mu<1\}$,

\item $\varphi_j:O_j\cap\partial \Omega \rightarrow
\{(y,\mu)\in \re^{n+1}:|y|<1,\mu=0\}$.
\end{itemize}

Let $\{\alpha_j\}$ be a partition of unity on $\partial \Omega$
such that $0\leq \alpha_j \in C^{\infty}_c(O_j),
\:\sum_{j=1}^m \alpha_j=1$ in $\partial \Omega$. If $u$ is a function on $\partial
\Omega$ decompose $u=\sum_{j=1}^m u\alpha_j $ and define the function
$$(u\alpha_j)\circ\varphi_j^{-1}(y,0):=(u\alpha_j
)(\varphi_j^{-1}(y,0)),\quad \mbox{for every}\;(y,0)\in U\cap
\{\mu=0\}.$$
Since $\alpha_j$ has compact support in $O_j \cap\partial \Omega$,
the function $(u\alpha_j)\circ\varphi_j^{-1}(\cdot,0)$ has compact support in
$U\cap \{\mu=0\}$ and therefore we may consider
$((u\alpha_j)\circ\varphi_j^{-1})(\cdot,0)$ to be defined in $\re^n$ extending it
by zero out of $U\cap \{\mu=0\}$. 
Now we define
$$H^{1/2}(\partial \Omega):=\{u:(u\alpha_j)\circ\varphi_j^{-1}(\cdot,0)\in
H^{1/2}(\re^n),\:j=1,...,m\}$$ equipped with the norm
$$\left(\sum_{j=1}^m ||(u\alpha_j)\circ\varphi_j^{-1}(\cdot,0)||_{H^{1/2}(\re^n)}^2\right)^{\frac{1}{2}}.$$ 

All these norms are
independent of the choice of the
system of local maps $\{O_j,\varphi_j\}$ and of the partition of
unity $\{\alpha_j\}$, and are all equivalent to
$$||u||_{H^{1/2}(\partial \Omega)}:=\left(||u||^2_{L^2(\partial
\Omega)}+\int_{\partial \Omega}\int_{\partial
\Omega}\frac{|u(z)-u(\overline{z})|^2}{|z-\overline{z}|^{n+1}}d\sigma_{z}
d\sigma_{\overline{z}}\right)^{\frac{1}{2}}.$$

We recall now the classical extension result that we will use in the proof of Theorem~\ref{energy-est}.

\begin{proposition}\label{extension}
Let $\Omega=\re^{n+1}_+$ or $\Omega$ be a bounded subset of $\re^{n+1}$ with Lipschitz
boundary $\partial \Omega$, and let $w$ belong to $H^{1/2}(\partial
\Omega)$.

Then, there exists a Lipschitz extension $\widetilde{w}$ of $w$ in
$\overline \Omega$ such that

\begin{equation}\label{H-est-th}
\int_{\Omega}|\nabla \widetilde{w}|^2 \leq C
||w||^2_{H^{1/2}(\partial \Omega)},\end{equation} where $C$ is a
constant depending only on $\Omega$.
\end{proposition}

For the sake of completeness (and since the proof will be important in our next paper~\cite{C-Cinti2}) we give the proof of this proposition.

\begin{proof}[Proof of Proposition \ref{extension}]
{\em Case 1}. $\Omega=\re^{n+1}_+$. Let $\zeta$ be a function belonging to
$H^{1/2}(\re^n)$. We prove that there exists a Lipschitz extension
$\widetilde{\zeta}$ of $\zeta$ in $\overline{\re_+^{n+1}}$ such that
\begin{equation}\label{ext}\int_{\re^{n+1}_+}|\nabla \widetilde{\zeta}|^2dx d\lambda\leq
C\int_{\re^n}\int_{\re^n}\frac{|\zeta(x)-\zeta(\overline{x})|^2}{|x-\overline{x}|^{n+1}}dx
d\overline{x}.\end{equation}

Let $K(x)$ be a nonnegative $C^{\infty}$ function defined on $\re^n$ with
compact support in $B_1$ and such that $\int_{\re^n} K(x)dx=1$.
Define $\widetilde{K}(x,\lambda)$ on $\re^{n+1}_+$ by
$$\widetilde{K}(x,\lambda):=\frac{1}{\lambda^{n}}K\left(\frac{x}{\lambda}\right).$$
Then, since
\begin{equation}\label{int=1}\int_{\re^n}\widetilde{K}(x,\lambda)dx=1\;\;\mbox{for all}\;\lambda>0,\end{equation}
we obtain, differentiating with respect to $x_i$ and $\lambda$,
\begin{equation}\label{int-K}\int_{\re^n}\partial_{x_i}\widetilde{K}(x,\lambda)dx=0\;\;\mbox{and}\;\;\int_{\re^n}\partial_{\lambda}\widetilde{K}(x,\lambda)dx=0
\;\;\mbox{for all}\;\lambda>0.\end{equation}

In addition, for a constant $C$ depending only on $n$, we have  $$|\nabla \widetilde{K}(x,\lambda)|\leq
\frac{C}{\lambda^{n+1}}\quad \mbox{for all}\; (x,\lambda)\in \re^{n+1}_+.$$
This holds, since the support of $\widetilde K$ is contained in $\{|x|<\lambda\}$ and, in this set,
$$|\nabla_x \widetilde{K}(x,\lambda)|\leq
\frac{C}{\lambda^{n+1}}\;\;\mbox{and}$$
$$|\partial_{\lambda}K(x,\lambda)|=
\left|-\frac{n}{\lambda^{n+1}}K\left(\frac{x}{\lambda}\right)-\frac{1}{\lambda^n}\nabla
K\left(\frac{x}{\lambda}\right)\cdot
\frac{x}{\lambda^2}\right|\leq \frac{C}{\lambda^{n+1}}.$$
Now we define the extension $\widetilde{\zeta}$ as
$$\widetilde{\zeta}(x,\lambda)=\int_{\re^n}\widetilde{K}(x-\overline{x},\lambda)\zeta(\overline{x})d\overline{x},$$
and we show that this function satisfies (\ref{ext}). Note also that, by \eqref{int=1}, for every $\lambda \geq 0$
\begin{equation}\label{L2}
||\widetilde \zeta(\cdot,\lambda)||_{L^2(\re^n)}\leq ||\zeta||_{L^2(\re^n)}.\end{equation}

To show \eqref{ext}, observe that, by \eqref{int-K},
\begin{equation*}
\partial_{x_i}\widetilde{\zeta}(x,\lambda)
=\int_{\re^n}\partial_{x_i}\widetilde{K}(x-\overline{x},\lambda)\zeta(\overline{x})d\overline{x}
=\int_{\re^n}\partial_{x_i}\widetilde{K}(x-\overline{x},\lambda)
\{\zeta(\overline{x})-\zeta(x)\}d\overline{x},\end{equation*}
and thus
$$|\partial_{x_i}\widetilde{\zeta}(x,\lambda)|\leq C\int_{\{|x-\overline{x}|<\lambda\}}
\frac{|\zeta(\overline{x})-\zeta(x)|}{\lambda^{n+1}}d\overline{x}.$$
In the same way
$$ |\partial_{\lambda}\widetilde{\zeta}(x,\lambda)|\leq C \int_{\{|x-\overline{x}|<\lambda\}}
\frac{|\zeta(\overline{x})-\zeta(x)|}{\lambda^{n+1}}d\overline{x}.$$

Hence, by Cauchy-Schwarz,
\begin{eqnarray*} |\nabla
\widetilde{\zeta}(x,\lambda)|^2\leq C\int_{\{|x-\overline{x}|<\lambda\}}
\frac{|\zeta(\overline{x})-\zeta(x)|^2}{\lambda^{n+2}}d\overline{x},\end{eqnarray*}
and then
\begin{eqnarray*} \int_{\re^{n+1}_+}|\nabla \widetilde{\zeta}|^2 dx
d\lambda&\leq& C\int_0^{+\infty} d\lambda\int_{\re^n} dx \int
_{\{|x-\overline{x}|<\lambda\}}d\overline{x}
\frac{|\zeta(\overline{x})-\zeta(x)|^2}{\lambda^{n+2}}\\
&\leq& C\int_{\re^n}dx
\int_{\re^n}d\overline{x}\int_{\{\lambda>|x-\overline{x}|\}}d\lambda
\frac{|\zeta(\overline{x})-\zeta(x)|^2}{\lambda^{n+2}}
\\
&\leq&
C\int_{\re^n}\int_{\re^n}\frac{|\zeta(x)-\zeta(\overline{x})|^2}{|x-\overline{x}|^{n+1}}dx
d\overline{x}.\end{eqnarray*}

{\em Case 2}. Consider now the general case of a function $w$
belonging to $H^{1/2}(\partial\Omega)$, where $\Omega$ is a
bounded subset of $\re^{n+1}$ with Lipschitz boundary.

Using the partition of unity $\{\alpha_j\}$ introduced in the beginning of this section, we
write $w=\sum_{j=1}^m w \alpha_j$. Observe that, for every
$j=1,...,m$,
\begin{eqnarray}\label{*}\int_{\partial \Omega} \int_{\partial
\Omega}\frac{|(w\alpha_j)(z)-(w\alpha_j)(\overline{z})|^2}
{|z-\overline{z}|^{n+1}}d\sigma_{z} d\sigma_{\overline{z}}\leq C||w||^2_{H^{1/2}(\partial
\Omega)},\end{eqnarray} where all constants $C$ in the proof depend only on $\Omega$.
Indeed,
\begin{eqnarray}\label{**}
&&\int_{\partial \Omega} \int_{\partial
\Omega}\frac{|(w\alpha_j)(z)-(w\alpha_j)(\overline{z})|^2}
{|z-\overline{z}|^{n+1}}d\sigma_{z} d\sigma_{\overline{z}}  \nonumber \\
&& \hspace{2em}=\int_{\partial \Omega} \int_{\partial
\Omega}\frac{|(w\alpha_j)(z)-w(z)\alpha_j(\overline{z})+
w(z)\alpha_j(\overline{z})-(w\alpha_j)(\overline{z})|^2}
{|z-\overline{z}|^{n+1}}d\sigma_{z} d\sigma_{\overline{z}}\nonumber  \\
&& \hspace{2em}\leq  2\int_{\partial \Omega} \int_{\partial
\Omega}\frac{|\alpha_j(z)-\alpha_j(\overline{z})|^2|w(z)|^2}
{|z-\overline{z}|^{n+1}}d\sigma_{z} d\sigma_{\overline{z}} 
\\ && \hspace{3.5em}
+ 2 \int_{\partial \Omega} \int_{\partial
\Omega}\frac{|w(z)-w(\overline{z})|^2|\alpha_j(\overline{z})|^2}
{|z-\overline{z}|^{n+1}}d\sigma_{z} d\sigma_{\overline{z}}. \nonumber
\end{eqnarray}
The integral in \eqref{**} is bounded by $C||w||^2_{L^2(\partial \Omega)}$. Indeed, using that $\alpha_j$ is Lipschitz, we get that the integral in \eqref{**} is controlled by 
$$C\int_{O_j \cap \partial \Omega}|w(z)|^2 d\sigma_z
\int_{O_j \cap \partial \Omega}  \frac{d\sigma_{\overline z}}{|z-\overline z|^{n-1}}\leq C||w||^2_{L^2(\partial \Omega)},$$ 
where we have used spherical coordinates centered at $z$
(after flattening the boundary) in the last integral. From this, \eqref{*} follows.

We flatten the boundary $\partial \Omega$ using the local maps
$\varphi_j$ introduced in the beginning of this section, and consider the functions
$$\zeta_j(y,0):=(w\alpha_j)(\varphi_j^{-1}(y,0)),$$
which are defined for $(y,0) \in U \cap\{\mu=0\}.$ Now $\zeta_j(\cdot,0)$,
extended by 0 outside of $U \cap\{\mu=0\}$, is defined in all of
$\re^n$, and we are in the situation of case 1. We make the
extension $\widetilde{\zeta}_j$ of $\zeta_j$ as in case 1. Since $\alpha_j\in C_c^{\infty}(O_j)$, there exists a function $\beta_j\in C_c^{\infty}(O_j)$ such that $\beta_j\equiv 1$ in the support of $\alpha_j$. Thus $\beta_j(\widetilde \zeta_ j \circ \varphi_j)$, extended by zero outside of $O_j$, is well defined as a function in $\overline \Omega$ and agrees with $w\alpha_j=\beta_j w \alpha_j$ on $\partial \Omega$. We now define
$$\widetilde{w}=\sum_{j=1}^m\beta_j(\widetilde{\zeta}_j\circ \varphi_j) \quad 
\mbox{in}\:\: \overline \Omega,$$ 
which agrees with $w$ on $\partial \Omega$.

Observe that, since $\varphi_j$ is a bilipschitz map and $\alpha_j,\:\beta_j \in C_c^{\infty}(O_j)$ for every
$j=1,...,m$, we have
$$|\nabla \widetilde{w}|\leq C\sum_{j=1}^m \left\{|\nabla \beta_j|
|\widetilde{\zeta_j}\circ\varphi_j|+|\beta_j||(\nabla \widetilde{\zeta_j})\circ\varphi_j|\right\} ,$$ and thus
$$\int_{\Omega}
|\nabla \widetilde{w}|^2  \leq C
\sum_{j=1}^m  \left\{ ||\widetilde{\zeta_j}||^2_{L^2(B_1\times (0,1))} + \int_{\re^{n+1}_+}|\nabla
\widetilde{\zeta}_j|^2\right\}.$$ By \eqref{L2} and \eqref{ext} of case 1, we have for every $j=1,...m,$
\begin{eqnarray*}&&||\widetilde{\zeta_j}||^2_{L^2(B_1\times (0,1))}+\int_{\re^{n+1}_+}|\nabla \widetilde{\zeta}_j|^2
\leq C\left\{||\zeta_j||^2_{L^2(B_1)} + \int_{\re^n}\int_{\re^n}\frac{|\zeta_j(y)-\zeta_j(\overline y)|^2}{|y-\overline y|^{n+1}}dyd\overline y\right\}\\
&&\hspace{2em} \leq C\left\{||w||^2_{L^2(\partial\Omega)}+\int_{\re^n}\int_{\re^n}\frac{|\zeta_j(y)-\zeta_j(\overline y)|^2}{|y-\overline y|^{n+1}}dyd\overline y\right\}.\end{eqnarray*}  Finally, using that
$\varphi_j$ is a bilipschitz map for every $j=1,...,m$, the
definition of $\zeta_j$, and \eqref{*}, we get
\begin{eqnarray*}&&
\int_{\re^n}\int_{\re^n}\frac{|\zeta_j(y)-\zeta_j(\overline{y})|^2}{|y-\overline{y}|^{n+1}}dy
d\overline{y}\\
&&\hspace{2em} = 
\int_{B_1}\int_{B_1}\frac{|(w\alpha_j)(\varphi_j^{-1}(y,0))-(w\alpha_j)(\varphi_j^{-1}(\overline{y},0))|}
{|y-\overline{y}|^{n+1}}dy
d\overline{y}\\
&&\hspace{2em} \leq C \int_{O_j\cap \partial \Omega}\int_{O_j\cap \partial
\Omega}\frac{|(w\alpha_j)(z)-(w\alpha_j)(\overline{z})|}
{|\varphi_j(z)-\varphi_j(\overline{z})|^{n+1}}d\sigma_z
d\sigma_{\overline{z}}\\
&&\hspace{2em}\leq C\int_{O_j \cap \partial \Omega}\int_{O_j \cap\partial
\Omega}\frac{|(w\alpha_j)(z)-(w\alpha_j)(\overline{z})|}{|z-\overline{z}|^{n+1}}d\sigma_z
d\sigma_{\overline{z}}\leq  C||w||^2_{H^{1/2}(\partial
\Omega)}.
\end{eqnarray*}
\end{proof}

\begin{remark} Let $\overline w$ be the harmonic extension of $w$ in
$\Omega$. Since $\overline w$ is the extension with minimal $L^2(\Omega)$-norm of $\nabla \overline w$, then we have that
$$\int_{\Omega}|\nabla \overline w|^2 dx d\lambda \leq \int_{\Omega}|\nabla \widetilde w|^2 dx d\lambda
\leq C||w||_{H^{1/2}(\partial \Omega)}.$$
\end{remark}

We give now the proof of the crucial Theorem \ref{key}.

\begin{proof}[Proof of Theorem \ref{key}]
The proof consists of two steps.

{\em Step 1}. Suppose that 
$$A=Q_1=\{x\in \re^n:|x_i|<1\:\mbox{for all}\:i=1,...,n\}$$ is a
cube in $\re^n$, and that
$$\Gamma=\{x_n=0\}\cap Q_1,$$ where $x=(x',x_n)\in \re^{n-1}\times \re.$
 We may assume $c_0=1$ by replacing $w$ by $w/c_0$. By hypothesis we have that $|w|\leq 1$ in $A$ and that
\begin{equation}\label{bound}
\begin{cases} |D w(x)|\leq
1/\varepsilon & \mbox{for a.e.}\: x \in Q_1\;\mbox{with}\;|x_n|<\varepsilon  \\
|D w(x)|\leq 1/|x_n|\quad & \mbox{for a.e.}\: x \in
Q_1\;\mbox{with}\;|x_n|>\varepsilon.
\end{cases}\end{equation}  We need to estimate the $H^{1/2}$-norm of $w$ in $Q_1$, given by
\begin{eqnarray*}
||w||^2_{H^{1/2}(Q_1)}= ||w||^2_{L^2(Q_1)}+
\int_{Q_1}\int_{Q_1}\frac{|w(x)-w(\overline x)|^2}{|x-\overline
x|^{n+1}}dx d\overline x.
\end{eqnarray*}
All constants $C$ in step 1 depend only on $n$ and differ from line to line. In this step, we take $0<\varepsilon\leq 1/2$.

First observe that $||w||^2_{L^2(Q_1)}\leq 2^n$. Let $x\in
Q_1^+=\{x\in Q_1:x_n>0\}$ and let $R_x$ be a radius depending on
the point~$x$, defined by 
$$R_x=\begin{cases}\varepsilon&
\mbox{if}\;0<x_n<\varepsilon\\
x_n/2 & \mbox{if}\;\varepsilon<x_n<1.\end{cases}$$
To bound
$||w||_{H^{1/2}(Q_1)}$, we consider the two cases $\overline x\in
B_{R_x}(x)$ and $\overline x \notin B_{R_x}(x)$, as follows:
\begin{eqnarray*}&&\int_{Q_1^+}dx\int_{Q_1}d\overline x\frac{|w(x)-w(\overline x)|^2}{|x-\overline x|^{n+1}}=
\\
&& \hspace{2em}=\int_{Q_1^+}dx\int_{Q_1\cap B_{R_x}(x)}d\overline
x\frac{|w(x)-w(\overline x)|^2}{|x-\overline x|^{n+1}}
+\int_{Q_1^+}dx\int_{Q_1\setminus
B_{R_x}(x)}d\overline x\frac{|w(x)-w(\overline x)|^2}{|x-\overline
x|^{n+1}}\\
&& \hspace{2em}:= I_1+I_2.\end{eqnarray*}

We use $|w|\leq 1$ to bound $I_2$, and the gradient
estimate \eqref{bound} for $w$ to bound $I_1$. In both cases we use spherical
coordinates, centered at $x$, calling $r=|x-\overline x|$ the radial coordinate.
We have
\begin{eqnarray*}
I_2&\leq &\int_{Q_1^+}dx\int_{Q_1\setminus
B_{R_x}(x)}d\overline x\frac{4}{|x-\overline x|^{n+1}}
\leq C\int_{Q_1^+}dx\int_{R_x}^{2\sqrt n}dr \frac{1}{r^2}\\
&\leq&
C\int_{Q_1^+}\frac{1}{R_x}dx
=C\left(\int_0^{\varepsilon}\frac{1}{\varepsilon}
dx_n+\int_\varepsilon^1 \frac{2}{x_n}dx_n\right)\leq C|\log
\varepsilon|.\end{eqnarray*}

Next, we bound $I_1$. We have
\begin{eqnarray*}
I_1&=&\int_{Q_1^+}dx\int_{Q_1\cap B_{R_x}(x)}d\overline x\frac{|w(x)-w(\overline
x)|^2}{|x-\overline x|^{n+1}}=\int_{Q_1^+}dx\int_{Q_1\cap B_{R_x}(x)}d\overline x \frac{|D
w(y(x,\overline x))|^2}{|x-\overline x|^{n-1}},\\
\end{eqnarray*} where $y(x,\overline x)\in Q_1 \cap B_{R_x}(x)$ is a point of the
segment joining $x$ and $\overline x$.

Now, \eqref{bound} reads $|Dw(y)|\leq\min\{1/\varepsilon,1/|y_n|\}$ for a.e. $y\in Q_1$. We use the bound $|Dw(y)|\leq 1/\varepsilon$ when $0<x_n<\varepsilon$. For $\varepsilon<x_n<1$, since $y(x,\overline x)\in B_{R_x}(x)=B_{x_n/2}(x)$, we have $y_n(x,\overline x)\geq x_n-R_x=x_n/2$, and thus $|Dw(y(x,\overline x))|\leq 1/y_n(x,\overline x)\leq 2/x_n$.
Thus, using spherical coordinates centered at $x$,
\begin{eqnarray*}
I_1 &\leq& C\int_0^\varepsilon dx_n \int_0^{\varepsilon}dr
\frac{1}{\varepsilon^2}
+C\int_{\varepsilon}^1 dx_n \int_0^{x_n/2}dr\frac{4}{x_n^2}\\
&\leq& C+C\int_\varepsilon^1 \frac{1}{x_n}dx_n \leq C|\log
\varepsilon|. \end{eqnarray*}

Finally, for $x\in Q_1^-=\{x \in Q_1:x_n <0\}$ we proceed in the
same way, and thus we conclude the proof of step 1.

{\em Step 2}. Suppose now the general situation of the theorem: $A\subset \re^n$ is a bounded Lipschitz domain, or $A=\partial\Omega$, where $\Omega$ is an open bounded subset of $\re^{n+1}$
with Lipschitz boundary. Recall that $\Gamma \subset A$ is the boundary (relative to $A$) of a Lipschitz open (relative to $A$) subset $M$ of $A$. From now on, we denote by $B_{r}(p)$ the ball
in $\re^n$ or in $\re^{n+1}$ indifferently, since we are
considering together the cases $A\subset \re^n$ and $A=\partial
\Omega$ with $\Omega \subset \re^{n+1}$. We define a finite open
covering of $A$ in the following way. 

First, for every $p\in
\Gamma$, we choose a radius $r_p$ for which
there
exists a bilipschitz diffeomorphism $\varphi_{p}:{B}_{r_{p}}(p)\cap
A\rightarrow Q_1$, where $Q_1$ is the unit cube of $\re^n$, such that $\varphi({B}_{r_{p}}(p)\cap \Gamma)=\{x\in
Q_1:x_n=0\}$.

Let $\overline \Gamma$ be the closure of $\Gamma$ in $\re^n$ or $\re^{n+1}$. Only in the case $A\subset \re^n$, it may happen that $\overline \Gamma \setminus \Gamma \neq \emptyset$. In such case, for $p\in \overline \Gamma \setminus \Gamma$, there exists a radius $r_p$ and a bilipschitz diffeomorphism $\varphi_p: B_{r_p}(p)\rightarrow (-3,1)\times (-1,1)^{n-1}$ such that $\varphi_p(p)=(-1,0,...,0)$, $\varphi_p(B_{r_p}(p)\cap A)=Q_1=(-1,1)^{n}$ and $\varphi_p(B_{r_p}(p)\cap \Gamma)=Q_1\cap \{x_n=0\}$. Thus, these last two properties hold for $p\in \overline \Gamma \setminus \Gamma$, as for the points $p\in \Gamma$ treated before.

Since $\overline \Gamma$ is compact, we can cover it by a finite number $m$ of
open balls $B_{r_{p_i}/2}(p_i)$, $i=1,...,m$, with half the radius $r_{p_i}$.
We set
$A^{(1)}_{r_i/2}:={B}_{r_{p_i}/2}(p_i)\cap A$ and $A^{(1)}_{r_i}:={B}_{r_{p_i}}(p_i)\cap A$. Observe that the number $m$ of
balls and the Lipschitz constant of $\varphi_{p_i}$ depend only on $A$ and $\Gamma$, as all constants from now on.

Next, consider the compact set $\mathcal K:= \overline{A}\setminus
\bigcup_{i=1}^m A^{(1)}_{r_i/2}.$ For every $q\in \mathcal K$, take a radius $0<s_q\leq (2/3)\mbox{dist}(q,\Gamma)$ for which there exists a bilipschitz diffeomorphism $\varphi_q:B_{s_q}(q_j)\cap A \rightarrow Q_1$. This is possible both if $q\in A$ or if $q\in \partial A$. Cover $\mathcal K$ by $l$ balls $B_{s_{q_j}/2}(q_j)$, $j=1,...,l$, with center $q_j \in \mathcal K$ and half of the radius $s_{q_j}$. Set $A^{(2)}_{s_j/2}:={B}_{s_{q_j}/2}(q_j)\cap A$ and $A^{(2)}_{s_j}:={B}_{s_{q_j}}(q_j)\cap A$.

Thus,
$\{A_{r_i/2}^{(1)}, A_{s_j/2}^{(2)}\}$ is a finite open
covering of $A$. Set
$\varepsilon_0:=\min_{i,j}\{r_i/2,s_j/2,1/2\}.$ If $z$
and $\overline{z}$ are two points belonging to $A$ such that
$|z-\overline{z}|<\varepsilon_0$, then there exists a set $A^{(1)}_{r_i}$, or
$A^{(2)}_{s_j}$, such that both $z$ and $\overline{z}$ belong to
$A^{(1)}_{r_i}$, or to $A^{(2)}_{s_j}$. Hence we have
\begin{equation}\label{subset}\{(z,\overline{z})\in A\times A:|z-\overline{z}|<\varepsilon_0\} \subset\left(\bigcup_{i=1}^m A^{(1)}_{r_i} \times
A^{(1)}_{r_i}\right)\cup \left(\bigcup_{j=1}^l A^{(2)}_{s_j}
\times A^{(2)}_{s_j}\right).\end{equation} Observe that
\begin{equation}\label{dist-y}  
\mbox{dist}(y,\Gamma)\geq\mbox{dist}(q_j,\Gamma)-|y-q_j|\geq \frac{3}{2} s_{q_j}-s_{q_j} =
\frac{s_j}{2}\geq \varepsilon_0\:\:\mbox{for every}\:\:y\in A^{(2)}_{s_j}.
\end{equation}
Let $L>1$ be a bound for the Lipschitz norm of all the functions $\varphi_{p_1},...,\varphi_{p_m},\varphi_{p_1}^{-1},...,\varphi_{p_m}^{-1}.$

Now, let $w$ as in the statement of the theorem. Let us first treat the case $0<\varepsilon \leq 1/(2L).$ Since
\begin{eqnarray*}
\int_{A}d\sigma_z \int_{\{\overline{z}\in A:|z-\overline{z}|>\varepsilon_0\}}
d\sigma_{\overline z}\frac{|w(z)-w(\overline z)|^2}{|z-\overline
z|^{n+1}}\leq \frac{4c^2_0}{\varepsilon_0^{n+1}}|A|^2=Cc_0^2,\end{eqnarray*}
we only need to bound the double integral in $\{z \in A\}\times \{\overline z \in A: |z-\overline z|< \varepsilon_0\}$. By \eqref{subset}, it suffices to bound the integrals in each $ A_{r_i}^{(1)}\times A_{r_i}^{(1)} $ and in each $ A_{s_j}^{(2)}\times A_{s_j}^{(2)} $.

Thus, for every $i$, consider
\begin{equation*}\int_{A^{(1)}_{r_i}}\int_{A^{(1)}_{r_i}}\frac{|w(z)-w(\overline
z)|^2}{|z-\overline z|^{n+1}}d\sigma_{z}d\sigma_{\overline
z}.\end{equation*} Recall that, by construction, there exists a
bilipschitz map $\varphi_{p_i}:A^{(1)}_{r_i}\rightarrow Q_1$
such that $\varphi_{p_i}(\Gamma \cap A^{(1)}_{r_i})=\{x\in Q_1:x_n=0\}.$
Thus, flattening the set $A^{(1)}_{r_i}$ using $\varphi_{p_i}$, we are
in the situation of step 1.
More precisely, since $\varphi_{p_i}$ is bilipschitz, we have that
$$\int_{A^{(1)}_{r_i}}\int_{A^{(1)}_{r_i}}\frac{|w(z)-w(\overline
z)|^2}{|z-\overline z|^{n+1}}d\sigma_{z}d\sigma_{\overline
z}\leq C\int_{Q_1}\int_{Q_1}\frac{|w_i(x)-w_i(\overline
x)|^2}{|x-\overline x|^{n+1}}d{x}d{\overline
x},$$ where we have set $w_i=w\circ \varphi_{p_i}^{-1}$. 

Given $x \in Q_1$, let $y=\varphi_{p_i}^{-1}(x)\in A_{r_i}^{(1)}.$ Recalling the definition of the Lipschitz constant $L$ above, we have $|x_n|\leq L\mbox{dist}(y,\Gamma)$ and hence
\begin{eqnarray*}
|Dw_i(x)|&\leq& L|Dw(y)|\leq Lc_0\min\left\{\frac{1}{\varepsilon},\frac{1}{\mbox{dist}(y,\Gamma)}\right\}\\
&\leq& Lc_0 \min\left\{\frac{1}{\varepsilon},\frac{L}{|x_n|}\right\}=L^2c_0\min\left\{\frac{1}{\varepsilon L},\frac{1}{|x_n|}\right\}.
\end{eqnarray*}
Thus we can apply the result proved in Step 1, with $\varepsilon$ replaced by $\varepsilon L$ (note that we have $\varepsilon L \leq 1/2$, as in Step 1), to the function $w_i/[(1+L^2)c_0].$ We obtain the desired bound $Cc_0^2 |\log(\varepsilon L)|\leq Cc_0^2 |\log\varepsilon|$.

Next, we consider the double integral in $A_{s_j}^{(2)} \times A_{s_j}^{(2)} $, for any $j\in \{1,...,l\}$.
Recall that there exists a bilipschitz diffeomorphism $\varphi_{q_j}:A_{s_j}^{(2)}\rightarrow Q_1$. Thus
$$\int_{A^{(2)}_{s_j}}\int_{A^{(2)}_{s_j}}\frac{|w(z)-w(\overline
z)|^2}{|z-\overline z|^{n+1}}d\sigma_{z}d\sigma_{\overline
z}\leq C\int_{Q_1}\int_{Q_1}\frac{|v_j(x)-v_j(\overline
x)|^2}{|x-\overline x|^{n+1}}d\sigma_{x}d\sigma_{\overline
x},$$
where now $v_j:=w\circ \varphi_{q_j}^{-1}$. By \eqref{dist-y} and \eqref{bound_grad}, $|Dw(y)|\leq c_0/\varepsilon_0$ a.e. in $A_{s_j}^{(2)}$, and $|D v_j|\leq C$ a.e. in $Q_1$. From this, the last double integral is bounded by
$$C\int_{Q_1}dx\int_{Q_1}\frac {d\overline x}{|x-\overline x|^{n-1}}\leq C\int_{Q_1}dx \int_0^{2 \sqrt n}dr\leq C.$$
This conclude the proof in case $\varepsilon \leq 1/(2L)$. 

Finally, given $\varepsilon \in (0,1/2)$ with $\varepsilon > 1/(2L)$, since \eqref{bound_grad} holds with such $\varepsilon$, it also holds with $\varepsilon$ replaced by $1/(2L)$. By the previous proof with $\varepsilon$ taken to be $1/(2L)$, the energy is bounded by $C|\log (1/(2L))|\leq C\leq C|\log\varepsilon|$ since $\varepsilon<1/2$. \end{proof}

\section{Energy estimate for global minimizers}
In this section we give the proof of Theorem \ref{energy-est}. 
It is based on a comparison argument. The proof can be resumed
in 3 steps. Let $v$ be a global minimizer of $\eqref{eq2}$.
\begin{enumerate}
\item [i)] Construct a comparison function $\overline{w}$, harmonic in $C_R$, which
takes the same values as $v$ on $\partial^+ C_R=\partial C_R \cap \{\lambda >0\}$ and thus, by
minimality of $v$,
$$\mathcal E_{C_R}(v)\leq \mathcal E_{C_R}(\overline{w}).$$
\item [ii)] Use estimate \eqref{H-est}:
$$\int_{C_R}|\nabla \overline{w}|^2 \leq C||\overline w||^2_{H^{1/2}(\partial C_R)}.$$ 
\item [iii)] Establish, using Theorem \ref{key}, the key estimate
$$||\overline w||^2_{H^{1/2}(\partial C_R)}\leq C R^{n-1}\log R.$$
\end{enumerate}

\begin{proof}[Proof of Theorem \ref{energy-est}]
Let $v$ be a bounded global minimizer of \eqref{eq2}. Let $u$ be its trace on $\partial \re^{n+1}_+$. Recall the definition \eqref{c_u} of the constant $c_u$. Let $s\in [\inf u, \sup u]$ be such that $G(s)=c_u$.

Through the proof, $C$ denotes positive constants depending only on $n$, $||f||_{C^1([\inf u,\sup u])}$ and $||u||_{L^\infty(\re^n)}$. As explained in \eqref{grad_CSM}, $v$ satisfies the
following bounds:
\begin{equation}\label{bounds-v}|v|\leq C\quad \mbox{and}\quad |\nabla v(x,\lambda)|\leq
\frac{C}{1+\lambda}\quad \mbox{for every}\:\:x\in
\re^n,\;\;\lambda \geq 0.\end{equation} 

We estimate the energy $\mathcal{E}_{C_R}(v)$ of $v$
using a comparison argument.  We define a function $\overline w=\overline{w}(x,\lambda)$ on
$C_R$
 in the following way. First we define $\overline{w}(x,0)$
on the base of the cylinder to be equal to a smooth function $g(x)$ which is
identically equal to $s$ in $B_{R-1}$ and $g(x)=v(x,0)$ for
$|x|=R$. The function $g$ is defined as follows:
\begin{equation}\label{g}
 g=s\eta_R+(1-\eta_R)v(\cdot,0),
\end{equation}
where $\eta_R$ is a smooth function depending only on $r=|x|$ such that $\eta\equiv 1$ in $B_{R-1}$ and $\eta\equiv 0$ outside $B_R$. Thus, $g$ satisfies 
\begin{equation}\label{g1} g \in [\inf u, \sup u]\quad \mbox{and}\quad |\nabla g|\leq C\:\:\mbox{in}\:\:B_R.\end{equation}
Then we define $\overline{w}(x,\lambda)$ as the unique
solution of the Dirichlet problem
\begin{equation}\label{w}\begin{cases}
\Delta \overline{w}=0 & \mbox{in} \;C_R\\
\overline{w}(x,0)=g(x)& \mbox{on}\;B_R\times \{\lambda=0\}\\
 \overline{w}(x,\lambda)=v(x,\lambda)& \mbox{on}\;\partial
 C_R\cap \{\lambda >0\}.\end{cases}\end{equation}

Since $v$ is a global minimizer of $\mathcal{E}_{C_R}$ and
$\overline{w}=v$ on $\partial C_R\cap \{\lambda >0\}$, then
\begin{eqnarray}&&\int_{C_R}\frac{1}{2}|\nabla v|^2dx d\lambda
+\int_{B_R}\{G(u)-c_u\}dx \nonumber \\
&& \hspace{2em}\leq \int_{C_R}\frac{1}{2}|\nabla \overline{w}|^2dx
d\lambda +\int_{B_R}\{G(\overline{w}(x,0))-c_u\}dx.\nonumber
\end{eqnarray}
We prove next that
$$\int_{C_R}\frac{1}{2}|\nabla \overline{w}|^2dx d\lambda +\int_{B_R}\{G(\overline{w}(x,0))-c_u\}dx \leq C R^{n-1}\log R.$$

Observe that the potential energy is bounded by $CR^{n-1}$. Indeed,
by definition $\overline{w}(x,0)=s$ in $B_{R-1}$, and hence 
\begin{eqnarray*}&&\int_{B_R} \{G(\overline{w}(x,0))-c_u\}dx=\int_{B_R\setminus
B_{R-1}}\{G(g(x))-c_u\}dx\nonumber \\
&& \hspace{2em}\leq C |B_R\setminus B_{R-1}|\leq
CR^{n-1}.\end{eqnarray*}

Thus, we only need to bound the Dirichlet energy. First of all,
rescaling, we set
$$w_1(x,\lambda)=\overline{w}(Rx,R\lambda),$$ for
 $(x,\lambda)\in C_1=B_1\times(0,1)$. Set
$$\varepsilon=1/R.$$ 
Observe that
$$\int_{C_R}|\nabla \overline{w}|^2=CR^{n-1}\int_{C_1}|\nabla
w_1|^2.$$ Thus, we need to prove that
\begin{equation}\label{est-epsilon}\int_{C_1}|\nabla
w_1|^2 \leq C\log R=C|\log \varepsilon|.\end{equation}

Since $w_1$ is harmonic in $C_1$, Proposition \ref{extension} gives that
$$\int_{C_1}|\nabla w_1|^2dx d\lambda \leq C||w_1||_{H^{1/2}(\partial C_1)}.$$

To control $||w_1||_{H^{1/2}(\partial C_1)}$, we apply Theorem \ref{key} to ${w_1}_{|_{\partial C_1}}$ in $A=\partial C_1$, taking $\Gamma=\partial B_1\times \{\lambda=0\}$.

Since $|w_1|\leq C$, we only need to check \eqref{bound_grad} in $\partial C_1$. In the bottom boundary, $B_1 \times \{0\}$, this is simple. Indeed $w_1\equiv s$ in $B_{1-\varepsilon}$, and thus we need only to control $|\nabla w_1(x,0)|=\varepsilon^{-1}|\nabla g(Rx)|\leq C\varepsilon^{-1}$ for $|x|>1-\varepsilon$, by \eqref{g1}. Here $\mbox{dist}(x,\partial B_1)<\varepsilon$, and thus \eqref{bound_grad} holds here.

Next, to verify \eqref{bound_grad} in $\partial C_1 \cap \{\lambda>0\}$ we use that $\overline w=v$ here and that we know
$$|\nabla v(x,\lambda)|\leq \frac{C}{1+\lambda}\quad \mbox{for every}\:(x,\lambda)\in C_R,$$ 
as stated in \eqref{bounds-v}.
Thus, the tangential derivatives of $w_1$ in $\partial C_1 \cap \{\lambda>0\}$ satisfy
$$|\nabla w_1(x,\lambda)|\leq \frac{CR}{1+R\lambda}=\frac{C}{\varepsilon+\lambda}\leq C \min\left\{\frac{1}{\varepsilon},\frac{1}{\lambda}\right\}.$$
Since $\mbox{dist}((x,\lambda),\Gamma)\geq \lambda$ on $\partial C_1\cap \{\lambda >0\}$, ${w_1}_{|_{\partial C_1}}$
satisfies the hypotheses of Theorem \ref{key}. We conclude that
(\ref{est-epsilon}) holds.
\end{proof}

\section{Energy estimate for monotone solutions in $\re^3$}
The following lemma will play a key role in this section to establish the energy estimate for monotone solutions in dimension
$n=3$.

\begin{lemma}\label{lemma}
Let $f$ be a $C^{1,\beta}$  function, for some $0<\beta<1$, and
$u$ a bounded solution of equation (\ref{eq1}) in $\re^3$, such
that $u_{x_3}>0$. Let $v$ be the harmonic extension of $u$ in
$\re^{4}_+$. Set
$$\underline{v}(x_1,x_2,\lambda):=\lim_{x_3 \rightarrow
-\infty}v(x,\lambda)\;\;\mbox{and}\;\; \overline{v}(x_1,x_2,\lambda):=\lim_{x_3 \rightarrow +\infty}v(x,\lambda)
.$$

Then, $\underline{v}$ and $\overline{v}$ are solutions of \eqref{eq2} in $\re^3_+$, and each of them is either constant or
it depends only on $\lambda$ and one Euclidian variable in the
$(x_1,x_2)$-plane. As a consequence, each $\underline{u}=\underline v(\cdot,0)$ and $\overline{u}=\overline
v(\cdot,0)$ is either constant or 1-D.

Moreover, set $m=\inf
\underline{u}\leq \widetilde m =\sup
\underline{u}$ and $\widetilde{M}=\inf \overline{u}\leq M =\sup \overline{u}$. 
Then, $G >G(m)=
G(\widetilde{m})$ in $(m,\widetilde{m})$,
$G'(m)=G'(\widetilde{m})=0$ and $G > G(\widetilde{M})=G(M)$ in $(\widetilde{M},M)$,
$G'(\widetilde{M})=G'(M)=0.$
\end{lemma}

\begin{proof}
The function $\overline v(x',\lambda)=\lim_{x_3 \rightarrow +\infty} v(x',x_3,\lambda)$ is the harmonic extension of $\overline{u}$. The key point
of the proof is to verify that $\overline v$ is a stable solution
of problem \eqref{eq2} in $\re^3_+$ and then apply Theorem 1.5 of
\cite{C-SM} on 1-D symmetry in $\re^3_+$.

The fact that $\overline v$ is a solution of problem \eqref{eq2} in
$\re^3_+$ is easily verified viewing $\overline v$ as a function
of $4$ variables, limit as $t\rightarrow +\infty$ of the solutions
$v^t(x',x_3,\lambda)=v(x',x_3+t,\lambda)$. By standard elliptic
theory, $v^t \rightarrow \overline v$ uniformly in the $C^2$ sense
on compact sets of $\overline{\re_+^{4}}$.

Now we prove that
$\overline{v}(x',\lambda)$ is a stable solution of problem
(\ref{eq2}) in $\re^3_+$. By Lemma 4.1 of \cite{C-SM}, the
stability of $\overline{v}$ is equivalent to the existence of a
function $\varphi>0$ in $\overline{\re^3_+}$ which satisfies
\begin{equation}\label{phi}\begin{cases}
\Delta \varphi=0 & \mbox{in} \;\re_+^{3} \\
 -\frac{\partial \varphi}{\partial \lambda}=f'(\overline{v})\varphi & \mbox{on}\;\partial
 \re^{3}_+.\end{cases}\end{equation}
To check the existence of $\varphi>0$ satisfying \eqref{phi}, we
use that $v_{x_3}>0$  in $\overline{\re^3_+}$ and that satisfies the problem
\begin{equation}\label{linear}\begin{cases}
\Delta v_{x_3}=0 & \mbox{in} \;\re_+^{4}\nonumber \\
 -\frac{\partial v_{x_3}}{\partial \lambda}=f'(v)v_{x_3} & \mbox{on}\;\partial
 \re^{4}_+.\end{cases}\end{equation}
This gives that $v$ is stable in $\re^4_+$, i.e.
\begin{equation}\label{stab-n+1}
\int_{\re^{4}_+}|\nabla \xi|^2 dx
d\lambda-\int_{\re^3}f'(v)\xi^2dx \geq 0,\quad \mbox{for
every}\;\;\xi \in C^\infty_c(\overline{\re_+^{4}}).\end{equation}

Next, we claim that
\begin{equation}\label{stab-n}
\int_{\re^{3}_+}|\nabla \eta|^2 dx'
d\lambda-\int_{\re^{2}}f'(\overline{v})\eta^2dx' \geq 0,\quad
\mbox{for every}\;\;\eta \in
C^\infty_c(\overline{\re_+^{3}}).\end{equation}

To show this, we take $\rho>0$ and $\psi_{\rho}\in
C^{\infty}(\re)$ with $0\leq \psi_{\rho} \leq 1,\:0\leq
\psi_{\rho}'\leq 2,\;\psi_{\rho}=0$ in $(-\infty, \rho) \cup
(2\rho+2, +\infty)$, and $\psi_\rho=1$ in $(\rho+1, 2\rho+1)$, and
we apply \eqref{stab-n+1} with
$\xi(x,\lambda)=\eta(x',\lambda)\psi_\rho(x_3).$ We obtain after
dividing the expression by $\alpha_\rho=\int \psi_\rho^2$, that
\begin{eqnarray*}&&\int_{\re^3_+}dx' d\lambda\ |\nabla \eta(x',\lambda)|^2
+\int_{\re^3_+} dx'd\lambda\ \eta^2(x',\lambda)\int_\re dx_3\
\frac{(\psi_\rho')^2(x_3)}{\alpha_\rho}\\
&&
\hspace{2em}-\int_{\re^{3}_+}dx'd\lambda\ \eta^2(x',\lambda)\int_\re dx_3\
f'(v(x',x_3,0))\frac{\psi^2_\rho(x_3)}{\alpha_\rho}\geq
0\end{eqnarray*} 
Passing to the limit as $\rho \rightarrow
+\infty$, and using $f \in C^1$ and that $v(x',x_3,\lambda)
\rightarrow \overline v(x',\lambda)$ as $x_3 \rightarrow +\infty$
uniformly in compact sets of $\overline{\re^4}_+$, we obtain  \eqref{stab-n}.

Since $\overline{v}(x',\lambda)$ is a stable solution of problem
(\ref{eq2}) in $\re^3_+$, by Theorem 1.5 (point b) in
\cite{C-SM} we deduce that $\overline{v}$ is constant or
$\overline{v}$ depends only on $\lambda$ and one Euclidian variable in the $x'$-plane. Now note that the
function $2(\overline{v}-\widetilde{M})/(M-\widetilde{M})-1$ is a
layer solution for a new nonlinearity. Using Theorem 1.2 a) of
\cite{C-SM}, which characterizes the nonlinearities $f$ for which
there exists a layer solutions for problem \eqref{eq2} in
dimension $n=1$, and restating the conclusion for $\overline{v}$,
we get $G'(\widetilde{M})=G'(M)=0$ and $G>G(\widetilde{M})=G(M)$
in $(\widetilde{M},M)$.

In the same way, we prove that the conclusion holds for $\underline v$ and that $G'(\widetilde{m})=G'(m)=0$ and $G>G(\widetilde{m})=G(m)$
in $(m,\widetilde{m})$.
\end{proof}
\begin{remark}
We claim that in the case of Allen-Cahn type equations, we could
prove the energy estimate \eqref{en3} for monotone solutions in dimension $n=3$
using the same argument as in the proof of Theorem
\ref{energy-layer}. The only difficulty is that in this section we do not assume $\lim_{x_3\rightarrow +\infty}v=1$, and then we do not know if $\lim_{T\rightarrow \infty}
\mathcal{E}_{C_R}(v^T)=0$ (see \eqref{lim-en-t} in the proof of Theorem \ref{energy-layer}). Using Lemma \ref{lemma}, we have that
$\overline{u}(x_1,x_2)=\lim_{x_3\rightarrow +\infty}u(x_1,x_2,x_3)$
is either a constant or it depends only on one variable. Then,
applying Theorem 1.6 of \cite{C-SM}, which gives the energy bounds
for 1-D solutions, we deduce that $\lim_{T\rightarrow +\infty}\mathcal E_{C_R}(v^T)\leq CR^2\log R$, and this is enough to carry out the proof of Theorem \ref{energy-layer} in the present setting. \end{remark} Before giving the
proof of Theorem \ref{energy-dim3}, we need the following
proposition. It is the analog of Theorem 4.4 of \cite{AAC} and asserts that the monotonicity of a solution implies its minimality among a suitable family of functions.

\begin{proposition}\label{monot-min}
Let $f$ be any $C^{1,\beta}$ nonlinearity, with $\beta \in
(0,1)$. Let $u$ be a bounded solution of (\ref{eq1}) in $\re^n$
such that
$u_{x_n}>0$, and let $v$ be its harmonic extension in $\re_+^{n+1}$.\\
Then, \begin{eqnarray*}\int_{C_R}\frac{1}{2}|\nabla
v(x,\lambda)|^2
dx d\lambda&+&\int_{B_R}G(v(x,0))dx \\
&\leq &\int_{C_R}\frac{1}{2}|\nabla w(x,\lambda)|^2 dx
d\lambda+\int_{B_R}G(w(x,0))dx,\end{eqnarray*} for every $w\in
C^1(\overline{\re^{n+1}_+})$ such that $w=v$ on $\partial^+C_R=\partial C_R \cap \{\lambda>0\}$ and $\underline{v}\leq
w\leq \overline{v}$ in $C_R$, where $\underline{v}$ and
$\overline{v}$ are defined by
$$\underline{v}(x',\lambda):=\lim_{x_n \rightarrow
-\infty}v(x',x_n,\lambda)\;\;\mbox{and}\;\; \overline{v}(x',\lambda):=\lim_{x_n \rightarrow +\infty}v(x',x_n,\lambda)
.$$
\end{proposition}

\begin{proof}
This property of minimality of monotone solutions among functions $w$ such that
$\underline{v}\leq w\leq \overline{v}$ follows from the
following two results:

\noindent
i) Uniqueness of solution to the problem
\begin{eqnarray}\label{eq2-ball}
\begin{cases}
\Delta w=0& \text{in}\; C_R,\\
w=v& \text{on}\; \partial^+C_R=\partial C_R\times\{\lambda>0\},\\
- \partial_{ \lambda}w=f(w)& \text{on}\; \partial^0 C_R=B_R\times\{\lambda =0\},\\
\underline{v}\leq w\leq \overline{v}&\text{in}\; C_R.
\end{cases}
\end{eqnarray}
Thus, the solution must be $w\equiv v$.
This is the analog of Lemma 3.1 of \cite{C-SM}, and below we comment on its proof.

\noindent ii) Existence of an absolute minimizer for
$\mathcal{E}_{C_R}$ in the set
$$C_v=\{w\in H^1(C_R)\ :\ w \equiv
v\:\mbox{on}\:\partial^+C_R,\:\underline{v}\leq w\leq
\overline{v}\:\mbox{in}\:\:C_R\}.$$
This is the analog of Lemma 2.10 of
\cite{C-SM}.

The statement of the proposition follows from the fact that by i) and ii), the monotone solution $v$, by uniqueness, must agree with
the absolute minimizer in $C_R$.

To prove points i) and ii), we proceed exactly as in \cite{C-SM},
with the difference that here we do not assume $\lim_{x_n\rightarrow \pm \infty}v=\pm 1$.
We have only to substitute $-1$ and
$+1$ by $\underline{v}$ and $\overline{v}$, respectively, in the
proofs of Lemma 3.1 and Lemma 2.10 in \cite{C-SM}. For this, it is important that $\underline v$ and $\overline v$ are, respectively, a strict subsolution and a strict supersolution of the Dirichlet--Neumann mixed problem \eqref{eq2-ball}.
We make a short
comment about these proofs.

The proof of uniqueness is based, as in Lemma 3.1 of \cite{C-SM}, on sliding the
function $v(x,\lambda)$ in the direction $x_n$. We set
$$v^t(x_1,...,x_n,\lambda)=v(x_1,...,x_n+t,\lambda)\quad \mbox{for every}\:(x,\lambda)\in \overline{C}_R.$$
Since $v^t\rightarrow \overline{v}$ as $t\rightarrow +\infty$
uniformly in $\overline{C}_R$ and $\underline{v}< w<
\overline{v}$, then $w<v^t$ in $\overline{C}_R$, for $t$ large
enough. We want to prove that $w<v^t$ in $\overline{C}_R$ for
every $t>0$. Suppose that $s>0$ is the infimum of those $t>0$ such
that $w<v^t$ in $\overline{C}_R$. Then by
applying the maximum principle and Hopf's lemma we get a contradiction, since one would have $w\leq v^s$ in $\overline C_R$ and $w=v^s$ at some point in $\overline C_R \setminus \partial^+ C_R$.

To prove the existence of an absolute minimizer for
$\mathcal{E}_{C_R}$ in the convex set $C_v$, we proceed exactly as
in the proof of Lemma 2.10 of \cite{C-SM}, substituting $-1$ and
$+1$ by the subsolutions and supersolution $\underline{v}$ and $\overline{v}$, respectively.
\end{proof}

We give now the proof of the energy estimate in dimension 3 for
monotone solutions without the limit assumptions.

\begin{proof}[Proof of Theorem \ref{energy-dim3}]
We follow the proof of Theorem 5.2 of \cite{AAC}. We need to prove that the comparison function $\overline w$, used in the proof of Theorem \ref{energy-est}, satisfies $\underline v\leq \overline w\leq \overline v$. Then we can apply Proposition
\ref{monot-min} to make the comparison argument with the function $\overline w$ (as for global minimizers).
We recall that $\overline w$ is the solution of problem \eqref{w},
\begin{equation}\label{pb-w}\begin{cases}
\Delta \overline{w}=0 & \mbox{in} \;C_R\\
\overline{w}(x,0)=g(x)& \mbox{on}\;B_R\times \{\lambda=0\}\\
\overline{w}(x,\lambda)=v(x,\lambda)& \mbox{on}\;\partial
C_R\cap \{\lambda >0\},\end{cases}\end{equation}
where $g=s\eta_R+(1-\eta_R)v(\cdot,0)$. Thus, if we prove that $\sup \underline v\leq s \leq \inf \overline v$, then $\underline v\leq g\leq \overline v$ and hence $\underline v$ and $\overline v$ are respectively, subsolution and supersolutions of \eqref{pb-w}. It follows that $\underline v\leq \overline w\leq \overline v$, as desired. 

To show that $\sup \underline v\leq s \leq \inf \overline v$, let
$m=\inf u=\inf \underline{u}$ and $M=\sup u=\sup \overline{u}$,
where $\underline{u}$ and $\overline{u}$ are defined in Lemma
\ref{lemma}. Set $\widetilde{m}=\sup \underline{u}$ and
$\widetilde{M}=\inf \overline{u}$, obviously $\widetilde{m}$ and
$\widetilde{M}$ belong to $[m,M]$. By Lemma \ref{lemma},
$\underline{u}$ and $\overline{u}$ are either constant or monotone
1-D solutions; moreover,
\begin{equation}\label{1}
G>G(m)=G(\widetilde{m})\;\;\mbox{ in}\;
(m,\widetilde{m})\end{equation} in case $m<\widetilde{m}$ (i.e.
$\underline u$ not constant), and
\begin{equation}\label{2}G>G(M)=G(\widetilde{M})\;\;\mbox{ in}\;(\widetilde{M},M)\end{equation}
 in case $\widetilde{M}<M$ (i.e.
$\overline u$ not constant).

In all four possible cases (that is, each
$\underline{u}$ and $\overline{u}$ is constant or
one-dimensional), we deduce from (\ref{1}) and (\ref{2}) that
$\widetilde{m}\leq \widetilde{M}$ and that there exists $s\in
[\widetilde{m},\widetilde{M}]$ such that $G(s)=c_u$ (recall that
$c_u$ is the infimum of $G$ in the range of $u$). We conclude that
$$\sup \underline{u}=\sup\underline{v}\leq \widetilde{m}\leq s\leq
\widetilde{M}\leq\inf\overline{v}=\inf\overline{u}.$$ Hence, we can apply Proposition
\ref{monot-min} to make comparison argument with the function $\overline w$ and obtain the desired energy estimate.
\end{proof}

\section{1-D symmetry in $\re^3$}

In this section we present the Liouville result due to Moschini \cite{mosch} that we will use in
the proof of 1-D symmetry in dimension $n=3$. Set
$$\mathcal{F}=\left\{F:\re^+\rightarrow \re^+, F\;\;\mbox{is
nondecreasing
and}\;\;\int_2^{+\infty}\frac{1}{rF(r)}=+\infty\right\}.$$ 
Note that $\mathcal{F}$ includes the function $F(r)=\log r$.

\begin{proposition}[\cite{mosch}]\label{moschini}
Let $\varphi\in L_{{\rm loc}}^{\infty}(\re^{n+1}_+)$ be a positive
function. Suppose that $\sigma \in H^1_{{\rm loc}}(\re^{n+1}_+)$
satisfies \begin{equation}\label{hp-mosch}\begin{cases} -\sigma
\rm{div}(\varphi^2\nabla \sigma)\leq
0\;&\mbox{in}\;\;\re^{n+1}_+\\
-\sigma\partial_\lambda \sigma \leq 0
\;&\mbox{on}\;\;\partial\re^{n+1}_+\end{cases}\end{equation}
 in the weak sense.
 Let the following condition hold:
\begin{equation}\label{hp-mosch1}\limsup_{R\rightarrow
+\infty}\frac{1}{R^2F(R)}\int_{C_R}(\varphi
\sigma)^2dx<\infty\end{equation} for some $F\in
\mathcal{F}$.

Then, $\sigma$ is constant.

In particular, this statement holds with $F(R)=\log R$.
\end{proposition}

\begin{remark}
In \cite{mosch}, the author proves the previous result under the
assumption \begin{equation}\label{hp-mosch2}\sum_{j=0}^{+\infty}\frac{1}{F(2^{j+1})}=+\infty\end{equation} on $F$. This is
equivalent to $\int_2^{+\infty}(rF(r))^{-1}dr=+\infty.$ Indeed, since the function 
$j \mapsto F(2^{j+1})$ is nondecreasing, we have that
$$\sum_{j=3}^{+\infty}\frac{1}{F(2^{j+1})}\leq
\int_2^{+\infty}\frac{ds}{F(2^{s+1})}=\frac{1}{\log 2}\int_{8}^{\infty}\frac{dr}{rF(r)}
\leq 
\sum_{j=2}^{+\infty}\frac{1}{F(2^{j+1})}.$$ Thus,
\eqref{hp-mosch2} holds if and only if $F \in \mathcal {F}$.
\end{remark}

\begin{proof}[Proof of Proposition \ref{moschini}]
We present the proof following that of Theorem 5.1 of
\cite{mosch}, here in $C_R$ instead of $B_R$. Set $\partial^+ C_R:=\partial C_R \cap \{\lambda>0\}.$ Since
$\sigma$ satisfies (\ref{hp-mosch}), we have
\begin{equation}\label{mosch1}
\rm{div}(\sigma\varphi^2\nabla \sigma)\geq \varphi^2|\nabla
\sigma|^2.
\end{equation}
On the other hand
\begin{equation}\label{mosch2}
\int_{\partial^+C_R}\sigma \varphi^2 \frac{\partial
\sigma}{\partial \nu}ds \leq
\left(\int_{\partial^+C_R}\varphi^2|\nabla\sigma|^2
ds\right)^{\frac{1}{2}}\left(\int_{\partial^+C_R}(\varphi\sigma)^2
ds\right)^{\frac{1}{2}},\end{equation} where $\nu$ denotes the
outer normal vector on $\partial^+C_R$. Now, set, as
in \cite{mosch},
$$D(R)=\int_{C_R}\varphi^2 |\nabla \sigma|^2dx.$$
Integrating (\ref{mosch1}) over $C_R$, using that
$-\sigma\partial_\lambda \sigma \leq 0$ on the bottom boundary $\partial^0C_R=\partial
C_R \cap \{\lambda=0\}$, and using (\ref{mosch2}), we get
\begin{equation}\label{D(R)}D(R)\leq D'(R)^{\frac{1}{2}}\left(\int_{\partial^+
C_R}(\varphi \sigma)^2
ds\right)^{\frac{1}{2}},\end{equation} which is the analog of
(5.5) in \cite{mosch} on $\partial^+ C_R$ instead of
$\partial B_R$. 

Assume that $\sigma$ is not constant. Then,
there exists $R_0>0$ such that $D(R)>0$ for every $R>R_0$.
Integrating \eqref{D(R)} and using Schwarz inequality, we get
that, for every $r_2>r_1>R_0$,
\begin{eqnarray}\label{r_j}
&&(r_2-r_1)^2 \left(\int_{C_{r_2}\setminus
C_{r_1}}(\varphi \sigma)^2dx\right)^{-1}=
(r_2-r_1)^2 \left(\int_{r_1}^{r_2}dR\int_{\partial^+C_R}ds\ (\varphi \sigma)^2\right)^{-1}\nonumber \\
&&\hspace{2em}\leq \int_{r_1}^{r_2}dR \left(\int_{\partial^+C_R}ds\ (\varphi \sigma)^2\right)^{-1}\leq
\frac{1}{D(r_1)}-\frac{1}{D(r_2)}.\end{eqnarray} 
Next, choose
$r_2=2^{j+1}r_*$ and $r_1=2^{j}r_*$, for some $r_*>R_0$, for every
$j=0,...,N-1$. Using \eqref{hp-mosch1}, \eqref{r_j} and summing over $j$, we find
that
\begin{equation}\label{serie}
\frac{1}{D(r_*)}\geq
C\sum_{j=0}^{N-1}\frac{1}{F(2^{j+1}r_*)}.\end{equation} If $j_0$
is such that $r_*\leq 2^{j_0}$, then, by hypothesis on $F$,
$F(2^{j+1}r_*)\leq F(2^{j+j_0+1})$. Thus, by \eqref{hp-mosch2}, the sum in \eqref{serie}
diverges as $N\rightarrow \infty$ and hence $D(r_*)=0$ for every
$r_*>R_0$, which is a contradiction.
\end{proof}
We can give now the proof of the 1-D symmetry result.
\begin{proof}[Proof of Theorem \ref{degiorgi}]  Without loss of generality we can suppose $e=(0,0,1)$.
We follow the proof of Lemma 4.2 in
\cite{C-SM}.

First of all observe that both global minimizers and
monotone solutions are stable. Then, in both cases, by Lemma 4.1
in \cite{C-SM}, there exists a function $\varphi \in
C^1_{{\rm loc}}(\overline{\re^{4}_+})\cap C^2(\re^{4}_+)$ such that
$\varphi >0$ in $\overline{\re^{4}_+}$ and
\begin{equation*}\begin{cases}
\Delta \varphi=0 & \mbox{in} \;\re_+^{4}\\
 -\frac{\partial \varphi}{\partial \lambda}=f'(v)\varphi & \mbox{on}\;\partial
 \re^{4}_+.\end{cases}\end{equation*}
Note that, if $u$ is a monotone solution in the direction $x_3$,
then we can choose $\varphi=v_{x_3}$, where $v$ is the harmonic extension of
$u$ in the half space. For $i=1,2,3$ fixed, consider the function
$$\sigma_i=\frac{v_{x_i}}{\varphi}.$$ We prove that $\sigma_i$ is
constant in $\re^{4}_+$, using the Liouville result of Proposition \ref{moschini} and our energy estimate.

Since
$$\varphi^2\nabla \sigma_i=\varphi\nabla v_{x_i}-v_{x_i}\nabla
\varphi,$$ we have that
$$\mbox{div}(\varphi^2\nabla \sigma_i)=0\quad
\mbox{in}\;\;\re^{4}_+.$$ Moreover, the normal derivative
$-\partial_\lambda \sigma_i$ is zero on $\partial
\re^{4}_+$. Indeed,
$$\varphi^2\partial_\lambda\sigma_i=\varphi v_{\lambda
x_i}-v_{x_i}\varphi_\lambda=0$$ since both $v_{x_i}$ and $\varphi$
satisfy the same boundary condition
$$-\partial_{\lambda} v_{x_i}-f'(v)v_{x_i}=0,\;\;\;
-\partial_{\lambda} \varphi-f'(v)\varphi=0.$$
Now, using our energy
estimates (\ref{energy}) or \eqref{en3}, we have for $n=3$,
$$\int_{C_R}(\varphi\sigma_i)^2\leq \int_{C_R}|\nabla v|^2 \leq CR^2
\log R,\quad \mbox{for every} \;R>2.$$ Thus, using Proposition
\ref{moschini}, we deduce that $\sigma_i$ is constant for every
$i=1,2,3$, i.e.,
$$v_{x_i}=c_i \varphi \quad \mbox{for some constant}\;\;
c_i,\;\;\;\mbox{with}\:\:i=1,2,3.$$

We conclude the proof
observing that if $c_1=c_2=c_3=0$ then $v$ is constant. Otherwise
we have
$$
c_iv_{x_j}-c_jv_{x_i}=0\quad \mbox{for every}\:i\neq j,$$ and we
deduce that $v$ depends only on $\lambda$ and on the variable
parallel to the
vector $(c_1,c_2,c_3)$. Thus, $u(x)=v(x,0)$ is 1-D.
\end{proof}

\section{Energy estimate for saddle-shaped solutions}

In this section we prove that the energy estimate \eqref{energy} holds
also for some saddle solutions (which are known \cite{Cinti} not to be global minimizers in
dimensions $2m\leq 6$) of the problem
$$(-\Delta)^{1/2}u=f(u)\quad \mbox{in}\;\;\re^{2m}.$$
Here, we suppose that $f$ is balanced and bistable, that is $f$ satisfies
hypotheses \eqref{h11}, \eqref{h22}, and \eqref{h33}.

We recall that saddle solutions are even with respect to the
coordinate axes and odd with respect to the Simons cone, which is
defined as follows:
$$\mathcal{C}=\{x\in
\re^{2m}:x_1^2+...+x_m^2=x_{m+1}^2+...+x_{2m}^2\}.$$

If we set $$s=\sqrt{x_1^2+\dots + x_m^2} \quad \mbox{ and }\quad
t=\sqrt{x_{m+1}^2+\dots + x_{2m}^2},$$ then the Simons cone
becomes ${\mathcal C}=\{s=t\}$. We say that
a solution $u$ of problem \eqref{eq1} is a saddle solution if it satisfies the following 
properties:

\noindent a) 
 $u$ depends only on the variables $s$ and $t$. We write
$u=u(s,t)$; 

\noindent b)
 $u>0$ for $s>t$; 

\noindent c)
$u(s,t)=-u(t,s).$

In \cite{Cinti}, the second author proves the existence of a
saddle solution $u=u(x)$ to problem (\ref{eq1}), by proving the
existence of a solution $v=v(x,\lambda)$ to problem (\ref{eq2})
with the following properties:

\noindent
 a) $v$ depends only on the variables $s,\: t$ and
$\lambda$. We write $v=v(s,t,\lambda)$; 

\noindent b) $v>0$ for
$s>t$;

\noindent c) $v(s,t,\lambda)=-v(t,s,\lambda)$.

The proof of the existence of such function $v$ is simple and it
uses a non-sharp energy estimate. Next, we sketch the proof.

We use the following notations:
$$\mathcal{O}:=\{x\in \re^{2m}:s>t\}\subset \re^{2m},$$
$$\widetilde{\mathcal{O}}:=\{(x,\lambda)\in \re^{2m+1}_+:x\in
\mathcal O\}\subset \re^{2m+1}_+.$$

Note that
$$
\partial{\mathcal O}={\mathcal C}.
$$

Let $B_R$ be the open ball in $\re^{2m}$ centered at the origin
and of radius $R$. We will consider the open bounded
sets
$$
{\mathcal O}_R:={\mathcal O}\cap B_R=\{s>t,
 |x|^2=s^2+t^2<R^2\}\subset \re^{2m},$$
$$\quad \widetilde{\mathcal O}_R:={\mathcal O}_R \times (0,R),\quad \mbox{and}\quad
\widetilde{\mathcal{O}}_{R,L}:={\mathcal O}_R \times (0,L).$$ Note that
$$
\partial {\mathcal O}_R=({\mathcal C}\cap \overline{B}_R)\cup
(\partial{B_R}\cap {\mathcal O}).
$$
Moreover we define the set
$$
\widetilde{H}_{0}^1(\widetilde{{\mathcal O}}_{R,L})=\{v\in
H^1(\widetilde{{\mathcal O}}_{R,L}) : v\equiv 0
\:\:\mbox{on}\:\:\partial^+\widetilde{{\mathcal O}}_{R,L},\:
v=v(s,t,\lambda) \text{ a.e.}\}.
$$

\begin{proof}[Proof of Theorem \ref{saddle}]
The proof of existence of the saddle solution $v$ in
$\re^{2m+1}_+$ can
be resumed in three steps.

{\em Step a}).   For every
$R>0$, $L>0$ consider the minimizer $v_{R,L}$ of the energy functional

$$
{\mathcal E}_{\widetilde{{\mathcal
O}}_{R,L}}(v)=\int_{\widetilde{{\mathcal O}}_{R,L}} \frac{1}{2}|\nabla
v|^2+ \int_{\mathcal O_{R}}G(v) $$
among all functions belonging to the space
$\widetilde{H}_{0}^1(\widetilde{{\mathcal O}}_{R,L})$. The existence of such minimizer, that may be taken to satisfy $|v_{R,L}|\leq 1$ by hypothesis \eqref{h22}, follows by lower semicontinuity of the energy functional. The minimizer $v_{R,L}$ is a solution of the equation \eqref{eq2} written in the $(s,t,\lambda)$ variables and we can assume that $v_{R,L}\geq 0$ in ${\widetilde{{\mathcal O}}_{R,L}}$.

{\em Step b}).   Extend $v_{R,L}$ to $B_R \times (0,L)$ by odd
reflection with respect to $\mathcal C\times (0,L)$, that is, $v_{R,L}(s,t,\lambda)=-v_{R,L}(t,s,\lambda)$. Then, $v_{R,L}$ is a solution in $B_R \times (0,L)$. 

{\em Step c}).   Define $v$ as the limit of the sequence
$v_{R,L}$ as $R\rightarrow +\infty$, taking $L=R^{\gamma}\rightarrow +\infty$ with $1/2\leq\gamma<1$.  With the aid of a non-sharp energy estimate, verify that $v \not \equiv 0$ and, as a consequence, that $v$ is a saddle solution. This step could also be carried out using the sharp energy estimate that we prove next.

Here, it is important to observe that the solution $v$ constructed
in this way is not a global minimizer in $\re^{2m+1}_+$ (indeed it
is not stable in dimensions $2m=4,6$ by a result of \cite{Cinti}),
but it is a minimizer in $\widetilde{\mathcal O}$, or in other
words, it is a minimizer under perturbations vanishing on the
Simons cone. Next, we use this fact to prove the energy estimate ${\mathcal E}_{\widetilde{{\mathcal
O}}_{R}}(v)\leq CR^{2m-1}\log R$ in the set $\widetilde{\mathcal O}_R=\mathcal O_R\times (0,R)$, using a comparison
argument as for global minimizers.

As before, we want to construct a comparison function $\overline
w$ in $\widetilde{\mathcal O}_R$ which agrees with $v$ on
$\partial^+\widetilde{\mathcal O}_R$ and such that
\begin{equation}\label{en-w}{\mathcal E}_{\widetilde{{\mathcal
O}}_{R}}(\overline w)=\int_{\widetilde{{\mathcal O}}_{R}}
\frac{1}{2}|\nabla \overline w|^2+ \int_{{\mathcal O}_{R}}G(\overline w)\leq CR^{2m-1}\log
R.\end{equation}
We define the function $\overline w=\overline w(x,\lambda)=\overline w(s,t,\lambda)$ in $\widetilde{{\mathcal O}}_{R}$ in the following way.

First we define $\overline{w}(x,0)$
on the base ${\mathcal O}_{R} $ of $\widetilde{{\mathcal O}}_{R} $ to be equal to a smooth function $g(x)$ which is
identically equal to $1$ in ${\mathcal O}_{R-1}\cap \{(s-t)/\sqrt 2>1\}$ and $g(x)=v(x,0)$ on
$\partial {\mathcal O}_{R} $. The function $g$ is defined as follows:
\begin{equation}\label{g-saddle}
 g=\eta_R\min \left\{1,\frac{s-t}{\sqrt 2}\right\}+(1-\eta_R)v(\cdot,0),
\end{equation}
where $\eta_R$ is a smooth function depending only on $r=|x|=(s^2+t^2)^{1/2}$ such that $\eta_R\equiv 1$ in $\mathcal O_{R-1}$ and $\eta_R\equiv 0$ outside $\mathcal O_R$.
Let $\overline{w}=\overline{w}(x,\lambda)=\overline{w}(s,t,\lambda)$ be any Lipschitz function in the closure of $\widetilde{\mathcal O}_R$ (the precise function $\overline{w}$ 
will be chosen later) such that
\begin{equation}\label{w-saddle}\begin{cases}
\overline{w}(x,0)=g(x)& \mbox{on}\;{\mathcal O}_R\times \{\lambda=0\}\\
 \overline{w}(x,\lambda)=v(x,\lambda)& \mbox{on}\;\partial
 \widetilde{{\mathcal O}}_{R}\cap \{\lambda >0\}.\end{cases}\end{equation}

Since $v$ is a global minimizer of $\mathcal{E}_{\widetilde{{\mathcal O}}_{R}}$ and
$\overline{w}=v$ on $\partial \widetilde{{\mathcal O}}_{R}\cap \{\lambda >0\}$, then
\begin{eqnarray}&&\int_{\widetilde{{\mathcal O}}_{R}}\frac{1}{2}|\nabla v|^2dx d\lambda
+\int_{\mathcal O_R}G(u)dx \nonumber \\
&& \hspace{2em}\leq \int_{\widetilde{{\mathcal O}}_{R}}\frac{1}{2}|\nabla \overline{w}|^2dx
d\lambda +\int_{\mathcal O_R}G(\overline{w}(x,0))dx.\nonumber
\end{eqnarray}
We establish now the bound \eqref{en-w} for the energy $\mathcal{E}_{\widetilde{{\mathcal O}}_{R}}(\overline w)$ of $\overline w$.

Observe that the potential energy of $\overline w$ is bounded by $CR^{2m-1}$, indeed
\begin{eqnarray*}
\int_{\mathcal O_R}G(\overline{w}(x,0))dx&\leq & C\left|\mathcal O_{R-1}\cap \left\{\frac{s-t}{\sqrt 2}<1\right\}\right|+ C\left|\mathcal O_{R}\setminus \mathcal O_{R-1}\right|\\
&\leq & C\int_0^{R-1}\{(t+\sqrt 2)^{m}-t^m\}t^{m-1}dt + C R^{2m-1}\leq CR^{2m-1}.
\end{eqnarray*}

Next, we bound the Dirichlet energy of $\overline w$. 
First of all, as in the proof of the energy estimate for global minimizers, we rescale and
set
$$w_1(x,\lambda)=\overline w(Rx,R\lambda)\quad \mbox{for
every}\;\;(x,\lambda)\in \widetilde{\mathcal O}_{1}.$$

Thus, the Dirichlet energy of $\overline w$ in $\widetilde{\mathcal O}_{R}$, satisfies \begin{eqnarray*}\int_{\widetilde{{\mathcal O}}_{R}}
\frac{1}{2}|\nabla \overline w|^2
= CR^{2m-1}\int_{\widetilde{{\mathcal O}}_{1}}
\frac{1}{2}|\nabla w_1|^2.\end{eqnarray*}
Setting
$\varepsilon=1/R$, we need to prove that
\begin{equation}\label{rescaled}
\int_{\widetilde{{\mathcal O}}_{1}} \frac{1}{2}|\nabla w_1|^2\leq
C|\log\varepsilon|.\end{equation}
Set $s=|(x_1,...,x_m)|$ and $t=|(x_{m+1},...,x_{2m})|$, for every $x=(x_1,...,x_{2m})\in \mathcal O_1$.
We observe that
\begin{eqnarray*}
&&\int_{\widetilde{{\mathcal O}}_{1}} \frac{1}{2}|\nabla w_1|^2dx d\lambda =\\
&&\hspace{2em}=C\int_0^1d\lambda \int_{\{s^2+t^2<1, s>t\geq 0\}}\left\{(\partial_{s}w_1)^2+ (\partial_{t}w_1)^2+(\partial_{\lambda }w_1)^2\right\} s^{m-1} t^{m-1} ds dt \\
&& \hspace{2em}\leq C\int_0^1d\lambda \int_{\{s^2+t^2<1, s>t\geq 0\}}\left\{(\partial_{s}w_1)^2+(\partial_{t}w_1)^2+(\partial_{\lambda} w_1)^2\right\} ds dt .
\end{eqnarray*}
We can see the last integral as an integral in the set $$\{(s,t,\lambda)\in \re^3:s^2+t^2<1, s>t\geq 0,0<\lambda <1\}\subset \re^3_+.$$ 

We consider now $w_2$ the even reflection of $w_1$  with respect to $\{t=0\}$. We set 
$$ \begin{cases}
s=z_1\\
t=|z_2|,
\end{cases}
$$
and we define $w_2(z,\lambda)=w_2(z_1,z_2,\lambda):=w_1(s,t,\lambda)$ in the Lipschitz set $$\Omega=\{(z_1,z_2,\lambda):z_1^2+z_2^2<1, z_1>|z_2|, 0<\lambda <1\}\subset \re^3_+.$$
We have that
\begin{eqnarray*}&&\int_0^1 d\lambda\int_{\{s^2+t^2<1, s>t>0\}}\left\{(\partial_{s}w_1)^2+(\partial_{t}w_1)^2+(\partial_{\lambda}w_1)^2\right\} ds dt\\
&&\hspace{2em} \leq \int_0^1 d \lambda\int_{\{z_1^2+z_2^2<1, z_1>|z_2|\}}|\nabla w_2|^2 dz_1 dz_2 .\end{eqnarray*}

Next we apply Proposition \ref{extension} and Theorem \ref{key} to the function $w_2$ in $\Omega$. Observe that $\Omega$ is Lipschitz as a subset of $\re^3$, but it is not Lipschitz at the
origin if seen as a subset of $\re^{2m+1}$. We now take $w_2$ to be the harmonic extension in $\Omega\subset \re^3$ of the boundary values given by \eqref{w-saddle}, after rescaling by $R$ and doing even reflection with respect to $t=|z_2|=0$.
Since $w_2$ is harmonic in $\Omega\subset \re^3$, Proposition \ref{extension} gives that
$$\int_{\Omega}|\nabla w_2|^2dz_1dz_2 d\lambda \leq C||w_2||^2_{H^{1/2}(\partial \Omega)}.$$
To bound the quantity $||w_2||_{H^{1/2}(\partial \Omega)}$, we apply Theorem \ref{key} with $A=\partial \Omega$ and 
$$\Gamma=\left(\left\{z_1^2+z_2^2< 1, z_1=|z_2|\right\}\times \{\lambda=0\}\right)\cup \left(\left\{z_1^2+z_2^2=1, z_1>|z_2|\right\}\times \{\lambda=0\}\right).$$
Since $|w_2|\leq 1$, we need only to check \eqref{bound_grad} in $\partial \Omega$. By the definition of $w_2$, we have that $w_2(z,0)\equiv 1$ if $\mbox{dist}(z,\Gamma)>\varepsilon$, while for $\mbox{dist}(z,\Gamma)<\varepsilon$,  $$|\nabla w_2(z_1,z_2,0)|=|\nabla w_1(s,t,0)|=\varepsilon^{-1}|\nabla g(Rx,0)|\leq C\varepsilon^{-1}=C\min\{\varepsilon^{-1},(\mbox{dist}(z,\Gamma))^{-1}\}.$$ Moreover, as in the proof of Theorem \ref{energy-est}, to verify \eqref{bound_grad} in $\partial \Omega \cap \{\lambda >0\}$ we use that $\overline w\equiv v$ here and the gradient bound \eqref{grad_CSM} for $v$. 
Thus, 
$$|\nabla w_2(z_1,z_2,\lambda)|\leq \frac{CR}{1+R\lambda}=\frac{C}{\varepsilon+\lambda}\leq C\min\left\{\frac{1}{\varepsilon},\frac{1}{\lambda}\right\}.$$
Hence,
$w_2$ satisfies the hypotheses of Theorem \ref{key} and we conclude that \eqref{rescaled} holds.
\end{proof}

\end{document}